%% file: main.tex
\documentclass{article}
\usepackage{lipsum}
\usepackage{amsfonts}
\usepackage{graphicx}
\usepackage{epstopdf}
\usepackage{algorithmic}
\usepackage{eucal}
\usepackage{libertine}
\ifpdf
  \DeclareGraphicsExtensions{.eps,.pdf,.png,.jpg}
\else
  \DeclareGraphicsExtensions{.eps}
\fi
\usepackage{booktabs}
\usepackage[inline]{enumitem}
\setlist[enumerate]{leftmargin=.5in}
\setlist[itemize]{leftmargin=.5in}
\usepackage[title]{appendix}
\usepackage{amsfonts}
\usepackage{amssymb}
\usepackage{amsmath}
\usepackage{amsthm}
\usepackage{mathtools}
\usepackage{hyperref}
\usepackage{cleveref}
\usepackage{tikz-cd}
\usetikzlibrary{arrows}
\usepackage[all]{xy}
\usepackage{subfigure}
\newcommand{\R}{{\mathbb{R}}}
\newcommand{\Rl}{{(\mathbb R, \le)}}
\newcommand{\e}{{\epsilon}}
\newcommand{\pt}{{\textnormal{pt}}}
\newcommand{\Obj}{{\textnormal{Obj}}}
\newcommand{\Morph}{{\textnormal{Morph}}}
\newcommand{\Rcat}{\mathbf{R}}
\newcommand{\Cat}{\mathbf{C}}

\newcommand{\xrightarrowdbl}[2][]{%
  \xrightarrow[#1]{#2}\mathrel{\mkern-14mu}\rightarrow
}

\newtheorem{theorem}{Theorem}
\newtheorem{remark}{Remark}
\newtheorem{example}{Example}
\newtheorem{definition}{Definition}
\newtheorem{proposition}{Proposition}
\newtheorem{lemma}{Lemma}

\newcommand{\footremember}[2]{%
    \footnote{#2}
    \newcounter{#1}
    \setcounter{#1}{\value{footnote}}%
}
\newcommand{\footrecall}[1]{%
    \footnotemark[\value{#1}]%
}

\title{Rank-based persistence}
\date{}

\author{Mattia G. Bergomi\footremember{champalimaud}{Champalimaud Research, Champalimaud Centre for the Unknown, 1400-038 Lisbon, Portugal.
$\{$mattia.bergomi, pietro.vertechi $\}$@neuro.fchampalimaud.org}
\and Pietro Vertechi\footrecall{champalimaud}}
\usepackage{geometry}

\ifpdf
\hypersetup{
  pdftitle={Rank-based persistence},
  pdfauthor={Mattia G. Bergomi, and Pietro Vertechi}
}
\fi

\begin{document}

\maketitle
\begin{abstract}
  Persistence has proved to be a valuable tool to analyze real world data robustly. Several approaches to persistence have been attempted over time, some topological in flavor, based on the vector space-valued homology functor, other combinatorial, based on arbitrary set-valued functors. To unify the study of topological and combinatorial persistence in a common categorical framework, we give axioms for a generalized rank function on objects in a target category so that functors to that category induce persistence functions. We port the interleaving and bottleneck distances to this novel framework and generalize classical equalities and inequalities. Unlike sets and vector spaces, in many categories the rank of an object does not identify it up to isomorphism: to preserve information about the structure of persistence modules, we define colorable ranks, persistence diagrams and prove the equality between multicolored bottleneck distance and interleaving distance in semisimple Abelian categories. To illustrate our framework in practice, we give examples of multicolored persistent homology on filtered topological spaces with a group action and labeled point cloud data.
\end{abstract}


\section{Introduction}

Topological persistence offers valuable tools to give encompassing representations of the geometry and topology of sampled objects, even in high dimension. Moreover, persistent homology and its encoding via persistence diagrams are endowed with essential properties in data analysis, such as stability~\cite{cohen-steiner_stability_2007} and resistance to occlusions~\cite{di_fabio_mayervietoris_2011}. Equipped with these fundamental features, persistent homology has been successfully employed in a vast number of applications~\cite{ferri_persistent_2017}.

We provided a first generalized theory of persistence to concrete categories in~\cite{bergomi_beyond_2019}.
This first generalization allows one to define persistence in a very general setting, that includes not only topological spaces or weighted graphs but also arbitrary categories of presheaves. However, it fails to fully generalize the classical theory, for it does not show how to define persistence functions based on functors to the target category of vector spaces (such as the homology functors). The primary technique developed in~\cite{bergomi_beyond_2019} to define stable persistent functions (named coherent sampling) requires using finite sets as target category, thus failing to recover, for example, the study of higher persistent homology groups.

Here, we aim at providing a new categorical generalization, embracing both the classical theory and the framework described in~\cite{bergomi_beyond_2019}. With this aim in mind, we first decompose classical persistent homology into its basic ingredients:
\begin{enumerate*}
    \item A filtration in a source category $\mathbf{Top}$.
    \item A functor $H_k$ from the source category to a target category $\mathbf{FinVec}_\mathbb{K}$.
    \item A notion of rank in the target category (the dimension of the vector space).
\end{enumerate*}

Thereafter, we explore which axioms each of these ingredients must respect for the classical results on persistence diagrams, bottleneck and interleaving distances to hold.

Not only we establish a common generalization of the combinatorial~\cite{bergomi_beyond_2019} and topological approach to persistence~\cite{edelsbrunner_topological_2000}, but we also find examples of novel target categories, different from $\mathbf{FinSet}$ or $\mathbf{FinVec}_\mathbb{K}$, giving rise to persistence modules with structure. Of particular interest is the case of persistent group representations, which arises naturally in the study of filtrations of topological spaces or simplicial objects with a group action compatible with the filtering function. By coloring the resulting persistence diagram, we are able to recover a notion of similarity that respects the structure of our target category, e.g.,~the group action. This construction holds in any target semisimple Abelian category: we show examples arising from labeled point cloud datasets, relevant for instance in a machine learning context.

The paper is organized as follows. In~\cref{sec:rank} we determine what are the features of the functions cardinality of a set and dimension of a vector space that make them suitable as a notion of rank of an object in a target category. We lay down an axiomatic foundation for such rank functions in the general setting of regular categories and provide as an example the length of an object in an Abelian category, which naturally generalizes the dimension of a vector space. In~\cref{sec:categoricalpersistence}, we first show how a functor from an arbitrary category to a regular category equipped with a rank function defines a \textit{categorical persistence function}. Then, we port the classical notions (e.g. regular and critical value, tameness and cornerpoint multiplicity) to the categorical setting and use them to define persistence diagrams. We show how persistence diagrams relate to persistence modules in the case of a semisimple category. Finally, we discuss the notions of interleaving and bottleneck distance and prove the inequality between them, i.e. that the interleaving distance is always greater or equal than bottleneck, with great generality. Equality between the two distances requires additional assumptions: in~\cref{sec:multicoloredpersistence} we discuss how, in the case of a semisimple target category, one can color the persistence diagram: a bottleneck distance computed allowing only color-preserving bijections is then equal to the interleaving distance. We finally show examples of multicolored persistence in the case of topological spaces with group actions and labeled point clouds.

For the sake of readability, we provide and exemplify basic definitions of category theory in~\cref{sec:appendix}.

\section{Rank functions in regular categories}
  \label{sec:rank}

  Historically, there have been two different treatments of persistent homology, one associating to a map of topological spaces the cardinality of the image of a map of sets~\cite{frosini1992measuring}, the other the dimension of the image of a map of vector spaces~\cite{edelsbrunner2008persistent}. To unify them in a common framework, we introduce here the concept of ranked category, i.e. a regular category (\cref{def:regular_cat}) equipped with an integer-valued rank function on objects.

  The reason behind choosing to work with regular categories is that, by definition, every morphism $X\xrightarrow{\phi}Y$ in a regular category $\Rcat$ can be factored in $X\xrightarrowdbl{\varepsilon}Z\xhookrightarrow{\mu}Y$ such that $\phi = \mu\circ \varepsilon$, where $\mu$ is a monomorphism (\cref{def:monomorphism}) and $\varepsilon$ a regular epimorphism (\cref{def:reg_epimorphism}), which in turn gives a good notion of image of a morphism ($Z$ being the image of $X\xrightarrow{\phi} Y$). This notion of image will allow us to define persistence functions based on the \textit{rank} of the image of a morphism. As both monomorphisms and regular epimorphisms are preserved by pullbacks (\cref{def:pullback}), we will be able to prove classical properties of persistence by building appropriate diagrams.

\begin{definition}
    \label{def:rank}
Let $\mathbf{R}$ be a regular category. Given a lower-bounded function $r:\textnormal{Obj}(\mathbf R) \to \mathbb{Z}$, we say that $r$ is a rank function if:
\begin{enumerate}
    \item For any monomorphism $A \hookrightarrow B$, $r(A) \le r(B)$
    \item For any regular epimorphism $ B \twoheadrightarrow D$, $r(B) \ge r(D)$
    \item For any pullback square:
\[
\begin{tikzcd}
    A \arrow[hookrightarrow]{r}{\iota_1} \arrow[swap, twoheadrightarrow]{d}{\pi_1} & B \arrow[twoheadrightarrow]{d}{\pi_2} \\
  C \arrow[hookrightarrow]{r}{\iota_2} & D
\end{tikzcd}
\]
where $\iota_1, \iota_2$ are monomorphisms and $\pi_1, \pi_2$ are regular epimorphisms, the following inequality holds:
\[
    r(B)-r(A) \ge r(D) - r(C)
\]
\end{enumerate}
We say that a rank function $r$ is strict if the inequalities in conditions 1 and 2 are strict unless the morphisms are invertible.
If furthermore $\Rcat$ has an initial object $\emptyset$ and $r(\emptyset) = 0$, we say that $r$ is $0$-based.
A ranked category $(\Rcat, r)$ is simply a regular category $\Rcat$ equipped with a rank function $r$.
\end{definition}

The pullback requirement in the third condition is not necessary: this will prove useful in the following sections when working with functors that do not preserve pullback squares.

\begin{proposition}
    \label{prop:nopullback}
    Given a ranked category $(\mathbf R, r)$, for any commutative square (not necessarily pullback):

\[
\begin{tikzcd}
    A^\prime \arrow[rightarrow]{r}{\iota_1^\prime} \arrow[swap, rightarrow]{d}{\pi_1^\prime} & B \arrow[twoheadrightarrow]{d}{\pi_2} \\
  C \arrow[hookrightarrow]{r}{\iota_2} & D
\end{tikzcd}
\]
where $\iota_2$ is a monomorphism and $\pi_2$ is a regular epimorphism, the following inequality holds:
\[
    r(B)-r(A^\prime) \ge r(D) - r(C)
\]

\end{proposition}
\begin{proof}
We can build the pullback square:

\[
\begin{tikzcd}
    A \arrow[hookrightarrow]{r}{\iota_1} \arrow[swap, twoheadrightarrow]{d}{\pi_1} & B \arrow[twoheadrightarrow]{d}{\pi_2} \\
  C \arrow[hookrightarrow]{r}{\iota_2} & D
\end{tikzcd}
\]
where $\iota_1$ is a monomorphism (as it is pullback of a monomorphism) and $\pi_1$ is a regular epimorphism (as it is pullback of a regular epimorphism).

Therefore $r(B)-r(A) \ge r(D) - r(C)$. We have a natural monomorphism $A^\prime \hookrightarrow A$, therefore, by property 1:
\[
    r(B)-r(A^\prime) \ge r(B) - r(A) \ge r(D) - r(C)
\]
\end{proof} 

\begin{proposition}
    \label{prop:exact}
    If a functor $F: \mathbf Q \to \mathbf R$ preserves the image factorization, i.e. it preserves monomorphisms and regular epimorphisms, and $r$ is a rank function on $\mathbf R$, then $r\circ F : \textnormal{Obj}(\mathbf Q) \to \mathbb{Z}$ is a rank function on $\mathbf Q$.
\end{proposition}
\begin{proof}
    As $F$ preserves monomorphisms, given a monomorphism $A \hookrightarrow B$ we have a monomorphism $F(A) \hookrightarrow F(B)$ and therefore $r(F(A)) \le r(F(B))$. Similarly given a regular epimorphism $B \twoheadrightarrow D$ we have a regular epimorphism $F(B) \twoheadrightarrow F(D)$ so $r(F(B)) \ge r(F(D))$.
    Finally, given a pullback square:
\[
\begin{tikzcd}
    A \arrow[hookrightarrow]{r} \arrow[swap, twoheadrightarrow]{d} & B \arrow[twoheadrightarrow]{d} \\
  C \arrow[hookrightarrow]{r} & D
\end{tikzcd}
\]

\noindent We have a commutative square (not necessarily pullback):

\[
\begin{tikzcd}
    F(A) \arrow[hookrightarrow]{r} \arrow[swap, twoheadrightarrow]{d} & F(B) \arrow[twoheadrightarrow]{d} \\
    F(C) \arrow[hookrightarrow]{r} & F(D)
\end{tikzcd}
\]

\noindent By~\cref{prop:nopullback} we have $r(F(B)) - r(F(A)) \ge r(F(D)) - r(F(C))$
\end{proof}

This in particular applies to regular functors, i.e. functors that preserve regular epimorphisms and finite limits, as preserving limits implies preserving monomorphisms.

\subsection{Fiber-wise rank functions}

In this section we formalize the notion of fiber-wise rank function, i.e. a function respecting the assumptions of~\cref{def:rank}, whose behavior on regular epimorphisms can be determined from fibers on ``points'' (see~\cref{def:fiber}).

\begin{definition}
    \label{def:fiberwise}
    Given a regular category $\Rcat$ with terminal object $\pt$, we say that a function $r:\Obj(\Rcat) \to \mathbb Z$ is fiber-wise if, for all regular epimorphism $B\overset{\phi}\twoheadrightarrow D$, we have the following equality:
    \begin{equation}
        \label{eq:fiberwise}
        r(B) - r(D) = \sum_{\iota \in \textnormal{Hom}(\pt, D)} (r(B\times_D^\iota \pt)-r(\pt))
    \end{equation}
where the $B\times_D^\iota \pt$ realizes the pullback:
\[
\begin{tikzcd}
    B\times_D^\iota \pt \arrow[hookrightarrow]{r}{} \arrow[swap, twoheadrightarrow]{d}{} & B \arrow[twoheadrightarrow]{d}{\phi} \\
    \pt \arrow[hookrightarrow]{r}{\iota} & D
\end{tikzcd}
\]
\end{definition}
The conditions of~\cref{def:rank} become easier to prove in the case of fiber-wise functions.
\begin{proposition}
    \label{prop:rankfromfibers}
    Let $\mathbf{R}$ be a regular category with terminal object $\pt$ and $r:\textnormal{Obj}(\mathbf D) \to \mathbb{Z}$ a lower-bounded function such that:
    \begin{enumerate}
        \item For any monomorphism $A \hookrightarrow B$, $r(A) \le r(B)$
        \item For any regular epimorphism $ A \twoheadrightarrow \pt$, $r(A) \ge r(\pt)$
        \item $r$ is fiber-wise
    \end{enumerate}
Then $r$ defines a rank function on $\mathbf{R}$.
\end{proposition}
\begin{proof}
    We will prove that $r$ respects the assumptions of~\cref{def:rank}. It obviously respects the first assumption. It also respects the second as, given $A\twoheadrightarrow C$:
    \[
        r(A) - r(C) = \sum_{\iota \in \textnormal{Hom}(\pt, C)}(r(A\times_C^\iota \pt)-r(\pt))
    \]
    and the right-hand side is $\ge 0$ as it is a sum of nonnegative quantities. To verify assumption 3, let us consider an inclusion $\pt \xhookrightarrow{\iota} C$ and the following diagram, where all squares are pullback
    \[
\begin{tikzcd}
    A \times_C^\iota \pt \arrow[hookrightarrow]{r} \arrow[twoheadrightarrow]{d} & A \arrow[hookrightarrow]{r}{\iota_1} \arrow[swap, twoheadrightarrow]{d}{\pi_1} & B \arrow[twoheadrightarrow]{d}{\pi_2} \\
    \pt \arrow[hookrightarrow]{r}{\iota} & C \arrow[hookrightarrow]{r}{\iota_2} & D
\end{tikzcd}
\]
As the outermost square is pullback, we have an isomorphism $A\times_C^\iota \pt \simeq B\times_D^{\iota_2\circ\iota}$, thus
\[
\begin{aligned}
    r(A) - r(C) & = \sum_{\iota \in \textnormal{Hom}(\pt, C)}(r(A\times_C^\iota \pt)-r(\pt))
                = \sum_{\iota \in \textnormal{Hom}(\pt, C)}(r(B\times_D^{\iota_2\circ\iota}  \pt)-r(\pt)) \\
                & \le \sum_{\iota^\prime \in \textnormal{Hom}(\pt, D)}(r(B\times_D^{\iota^\prime} \pt)-r(\pt))
                =  r(B) - r(D)
\end{aligned}
\]
where the inequality comes from the fact that all summands are nonnegative and one sum has all the summands of the other plus potentially some more.
\end{proof}

Under the stronger assumptions of Abelian category (\cref{def:abelian_cat}), the fiber-wise condition simplifies greatly. As Abelian categories have a null object, given an epimorphism $B\overset{\phi}\twoheadrightarrow D$,~\cref{eq:fiberwise} is equivalent to $r(B)-r(D)=r(ker(\phi))-r(0)$, where of course $ker(\phi) \hookrightarrow B \overset{\phi}\twoheadrightarrow D$ is a short exact sequence.

\begin{proposition}
    \label{prop:fiberwiseshortsequence}
    Let $\mathbf R$ be an Abelian category. Then $r:\textnormal{Obj}(\Rcat) \to \mathbb Z$ is fiber-wise if and only if for all short exact sequence $A \hookrightarrow B \twoheadrightarrow D$, $r(A)+r(D)=r(B)+r(0)$.
\end{proposition}

In the Abelian case fiber-wise functions require less  assumptions to verify the rank properties:
\begin{proposition}
    \label{prop:abelianfiberwise}
    Let $\mathbf R$ be an Abelian category. If $r:\textnormal{Obj}(\Rcat) \to \mathbb Z$ is fiber-wise and for all $D \in \textnormal{Obj}(\Rcat)$, $r(0) \le r(D)$ then $r$ is a rank. Furthermore, if $r(0) = r(D)$ only if $D$ is null then $r$ is strict.
\end{proposition}
\begin{proof}
    We can embed any monomorphism or epimorphism in a short exact sequence $A \hookrightarrow B \twoheadrightarrow D$ where, as $r$ is fiber-wise, $r(0) + r(B) = r(A) + r(D)$. As $r(0) \le r(D)$, $r(B) \ge r(A)$. Similarly, as $r(0) \le r(A)$, $r(B) \ge r(D)$.

    To prove strictness, let us assume for example that $r(A) = r(B)$, then $r(D) = r(0)$ therefore $D$ is null so $A \simeq B$. Similarly if $r(B) = r(D)$ then $r(A) = r(0)$ therefore $A$ is null so $B \simeq D$.
\end{proof}

\subsubsection*{Examples of fiber-wise rank functions}

The cardinality function $|-| :\mathbf{FinSet} \to \mathbb{Z}$ is fiber-wise (the terminal object $\pt$ being the singleton). Indeed given a surjective map of sets $A \overset{f}{\twoheadrightarrow} D$:
\[
    |A|-|D| = \sum_{d \in D} (|f^{-1}(d)| - 1)
\]

$|-|$ is clearly a rank function: it is nondecreasing on monomorphisms and nonincreasing on epimorphisms. $|-|$ is also strict, as a monomoprhism (or an epimorphism) between two sets with the same number of elements is invertible.

The category $\mathbf{FinVec}_{\mathbb K}$ is regular and the dimension function is a strict fiber-wise rank. This case can be generalized to a wide variety of Abelian categories. We recall from~\cite[Sect.~1]{etingof_tensor_2015} that, in an Abelian category an object $X$ has finite length if there exists a series of inclusions:
\[
    0 \simeq X_0 \hookrightarrow X1 \hookrightarrow \dots \hookrightarrow X_n \simeq X
\]
where all quotients $X_i / X_{i-1}$ are simple. If such series exists, then $length(X) = n$
If all objects in an Abelian category have finite lenght, we say that the category has finite length.

The function $length$ is $0$-based and, by~\cite[Sect.~1]{etingof_tensor_2015}, for all short exact sequence $A \hookrightarrow B \twoheadrightarrow D$, we have $length(B) = length(A)+length(D)$ so, by~\cref{prop:fiberwiseshortsequence}, $length$ is fiber-wise.

\begin{proposition}
    \label{prop:length}
    Given $\mathbf D$ an Abelian category of finite length, the function $$length:\textnormal{Obj}(\mathbf D) \to \mathbb{Z}$$ is a strict $0$-based fiber-wise rank.
\end{proposition}
\begin{proof}
    $length$ is nonnegative, and $length(X) = 0$ if and only if $X$ is initial so, by~\cref{prop:abelianfiberwise}, it is a strict $0$-based fiber-wise rank.
\end{proof}

The rank function $length$ is is characterized by the following two features:
\begin{enumerate*}
    \item It is $0$-based and fiber-wise.
    \item It has value $1$ on simple objects.
\end{enumerate*}
Furthermore $length$ and $|-|$ share the additive property, i.e. given two objects $X, Y \in \textnormal{Obj}(\Rcat)$, $r(X \amalg Y) = r(X) + r(Y)$.

\begin{remark}
    Even though the category $\mathbf{FinMod}_\mathbb{Z}$ of finitely generated Abelian groups (i.e. finitely generated $\mathbb Z$-modules) does not have finite length, we have an image factorization-preserving functor $-\otimes_\mathbb{Z} \mathbb{Q}:\mathbf{FinMod}_\mathbb{Z}\to\mathbf{FinVec}_\mathbb{Q}$. By~\cref{prop:exact}, the rank function $dim:\Obj(\mathbf{FinVec}_\mathbb{Q}) \to \mathbb Z$ induces a rank function on $\mathbf{FinMod}_\mathbb{Z}$, which coincides with the rank of finitely generated Abelian groups.
\end{remark}

\section{Categorical persistence}
\label{sec:categoricalpersistence}

In general, given an arbitrary functor $\Psi$ from a source category $\Cat$ to a regular target category $\Rcat$ equipped with a rank $r$, we can not naturally define a rank on $\Cat$, unless $\Cat$ is regular and $\Psi$ preserves the image factorization (i.e. monomorphisms and regular epimorphisms), see~\cref{prop:exact}. Unfortunately, these assumptions do not hold in many common cases: for instance the category $\mathbf{Top}$ is not regular and, even though $\mathbf{FinSimp}$ is regular, no homology functor $H_k:\mathbf{FinSimp}\to\mathbf{FinVec}$ preserves the image factorization. However, we can still define an integer-valued function on the morphisms of $\Cat$, as a \textit{categorical persistence function}. While categorical persistence functions have very mild assumptions, they will be sufficient to guarantee the classical constructions and results of persistent homology. See~\cref{tab:analogies} for an intuitive comparison between the classical framework and ours.

\subsection{Categorical persistence functions}

In persistent homology, the functor $H_k$ maps a filtration of topological spaces and a field of coefficients $\mathbb K$, into a sequence of $\mathbb K$-vector spaces $V_u$ equipped with maps $V_u \to V_v$ for $u \le v$. The persistent homology group and persistent Betti number correspond to the image of $V_u \rightarrow V_v$ and its rank, respectively. The aim of this section is to extend this procedure to arbitrary categories.

First, we extend the notion of persistence function from~\cite{bergomi_beyond_2019}, which in turn generalizes persistent Betti number functions.
\begin{definition}
    \label{def:persistence}
    Let $\mathbf{D}$ be a category. We say that a lower-bounded function $p:\textnormal{Morph}(\mathbf{D}) \to\mathbb{Z}$ is a categorical persistence function if, for all $u_1 \to u_2 \to v_1 \to v_2$, the following inequalities hold:

\begin{enumerate}
    \item $p(u1\to v1) \le p(u2\to v1)$ and $p(u2\to v2) \le p(u2\to v1)$.
    \item $p(u2\to v1) - p(u1\to v1) \ge p(u2\to v2) - p(u1\to v2)$.
\end{enumerate}

\end{definition}

If $\mathbf D$ is the poset category $\Rl$ whose objects are real numbers, with a unique morphism from $u$ to $v$ if $u \le v$, then we recover the definition of persistence function from~\cite{bergomi_beyond_2019}. In some sense, this is a categorification~\cite{baez_categorification_1998} of that notion.

\begin{proposition}
    \label{prop:compose}
    Given a functor $F:\mathbf{C} \to \mathbf{D}$ and a categorical persistence function $p$ for $\mathbf D$, $p\circ F$ is a categorical persistence function for $\mathbf{C}$.
\end{proposition}
\begin{proof}
    Given $u_1 \to u_2 \to v_1 \to v_2$ we have, by functoriality, $F(u_1) \to F(u_2) \to F(v_1) \to F(v_2)$, so:
    \[
        (p\circ F)(u_1 \to v_1) = p(F(u_1 \to v_1)) \le p(F(u_2 \to v_1)) = (p\circ F)(u_2 \to v_1)
    \]
    All other inequalities are proved in an analogous way.
\end{proof}

\begin{remark}
  All functors we consider are covariant. Contravariant functors, if used, will be written in the covariant form $F:\Cat^{op}\to\mathbf{D}$.
\end{remark}

Classical persistent homology is defined in terms of dimensions of images of maps between vector spaces. The same construction holds in this setting. Given a regular category $\mathbf R$, we denote by $im: \textnormal{Morph}(\Rcat) \to \textnormal{Obj}(\Rcat)$ the map associating to each morphism its image. Given a rank function $r$ on $\Rcat$ and a functor $F: \mathbf{C} \to \mathbf{R}$, the function $r\circ im\circ F: \textnormal{Morph}(\mathbf C) \to \mathbb{Z}$ is a categorical persistence function. We will prove it in the following two propositions.

\begin{proposition}
    \label{prop:persistencefromrank}
    Given a ranked category $(\mathbf{R}, r)$, $r\circ im$ defines categorical persistence function on $\mathbf{R}$.
\end{proposition}
\begin{proof}
    Let us consider a diagram $u_1 \to u_2 \to v_1 \to v_2$ in $\mathbf D$. Then, we have an inclusion
    $$im(u1 \to v1) \hookrightarrow im(u_2 \to v_1),$$
    and thus
    $$r(im(u_1 \to v_1)) \ge r(im(u_2 \to v_1)).$$
    Similarly, we have an epimorphism $im(u_2 \to v_1) \twoheadrightarrow im(u_2 \to v_2)$, so $r(im(u_2 \to v_1))\ge r(im(u_2\to v_2))$.
    To prove the second condition of~\cref{def:rank},  we remind that the inequality of the third condition of the same definition holds for all commutative squares and not only pullback squares. Thus, we can build the following commutative diagram:
\[
\begin{tikzcd}
    im(u_1 \to v_1) \arrow[hookrightarrow]{r}{\iota_1} \arrow[swap, twoheadrightarrow]{d}{\pi_1} & im(u_2 \to v_1) \arrow[twoheadrightarrow]{d}{\pi_2} \\
    im(u_1 \to v_2) \arrow[hookrightarrow]{r}{\iota_2} & im(u_2 \to v_2)
\end{tikzcd}
\]
where $\iota_1, \iota_2$ are monomorphisms and $\pi_1, \pi_2$ are regular epimorphisms. By the third condition of~\cref{def:rank}, we have:
\[
    r(im(u_2 \to v_1))-r(im(u_1 \to v_1)) \ge r(im(u_2 \to v_2)) - r(im(u_1 \to v_2))
\]
\end{proof}

By combining~\cref{prop:persistencefromrank} and~\cref{prop:compose}, we obtain:
\begin{proposition}
    Given a ranked category $(\mathbf R, r)$ and a functor $F: \mathbf{C} \to \mathbf{R}$, the function $r\circ im\circ F: \textnormal{Morph}(\mathbf C) \to \mathbb{Z}$ is a categorical persistence function.
\end{proposition}

Functors to $(\mathbf{FinSet}, |-|)$ allow to recover persistent 0-Betti numbers, as well as all examples of coherent sampling in~\cite{bergomi_beyond_2019} such as blocks, edge-blocks and $\cal F$-connected components. In this framework , classical persistent homology can be seen as a combination of the functor $H_k:\mathbf{FinSimp} \to \mathbf{FinVec}_{\mathbb K}$ with the fiber-wise rank function dimension.

\begin{remark}[$\Rl$-indexed diagrams]\label{rem:indexeddiagram} Classically, persistent Betti numbers, as well as persistence functions in the sense of~\cite{bergomi_beyond_2019}, are defined on $\Delta^+$, i.e. on pairs $(u, v)\in\mathbb R^2$ with $u \le v$. Categorical persistence functions, on the other hand, are defined more abstractly on $\Morph(\Cat)$. However, as $\Delta^+$ is in one-to-one correspondence with $\Morph(\Rl)$, to define a function on $\Delta^+$ from a categorical persistence function in $\Cat$, we simply need a functor $F:\Rl \to \Cat$. We denote the category of these functors as $\mathbf C^\Rl$ and call them $\Rl$-indexed diagrams in $\mathbf C$. They are analogous to filtrations with the difference that, given a $\Rl$-indexed diagram $F$, we do not require morphisms $F(u)\to F(v)$ to be monomorphisms. As an example, given a topological space $X$ and a real-valued function $f:X\to \mathbb R$, the functor
\begin{gather*}
    F: \Rl \to \mathbf{Top} \\
    u \mapsto f^{-1}((-\infty, u])
\end{gather*}
is a $\Rl$-indexed diagram in $\mathbf{Top}$, with $F(u \le v)$ given by the inclusion $f^{-1}((-\infty, u]) \subseteq f^{-1}((-\infty, v])$. Similarly, the homology in degree $k$ of the various sublevels also naturally forms a $\Rl$-indexed diagram $u \mapsto H_k(f^{-1}((-\infty, u])))$, where morphisms are no longer necessarily injective.
\end{remark}
Refer again to~\Cref{tab:analogies} for an intuitive list of analogies between the classical and proposed frameworks.
\begin{table}
  \footnotesize{\caption{From the classical to the categorical framework.\label{tab:analogies}}}
\begin{center}
\begin{tabular}{@{}ll@{}}\toprule
 \textbf{Classical framework} & \textbf{Categorical framework}\\\midrule
 Topological spaces & Source category $\mathbf C$ \\
 Vector spaces & Regular target category $\mathbf R$\\
 Dimension & Rank function on $\mathbf R$\\
 Homology functor & Arbitrary functor from $\mathbf C$ to $\mathbf R$ \\
 Filtration of topological spaces & $\Rl$-indexed diagram in $\mathbf C$ \\
 \bottomrule
\end{tabular}
\end{center}
\end{table}

\subsection{Persistence diagrams}
\label{sec:persistencediagrams}
After generalizing the main ingredients of persistence, it is important to discuss how the notion of persistence diagram can be defined in this new context. Indeed, persistence diagrams are agile tools, that allow one to easily represent the features determined by the persistence function as a multiset of two-dimensional points. This representation is suitable for both rapid visualization and comparison of filtered objects.

In the following we will work with an arbitrary category $\mathbf C$, a categorical persistence function $p:\textnormal{Morph}(\mathbf C) \to \mathbb Z $, as well as a $\Rl$-indexed diagram $F$, and the induced persistence function on $\Delta^+$:
\begin{gather*}
    p_F: \Delta^+ \to \mathbb Z \\
    (u, v) \mapsto p(F(u \le v))
\end{gather*}

To define a persistence diagram we follow the approach given in~\cite{bergomi_beyond_2019}, which in turn draws from the definition of multiplicity of~\cite{damico_natural_2010} and~\cite{frosini_size_2001}. We will limit ourselves to the tame case: to do so we will need to generalize the definition of tameness from~\cite{bubenik_categorification_2014}.

\begin{definition} \label{def:tame}~\cite[Def. 4.3]{bubenik_categorification_2014}
    Let $F \in \mathbf C ^\Rl$. Let $I \subset \R$ be an interval. We say that $F$ is constant on $I$ if for all $a \le b \in I$ we have $p_F(a, a) = p_F(a, b) = p_F(b, b)$. We call $a \in \R$ a regular value (resp. right- or left-regular) for $F$ if there is a connected neighborhood (resp. connected right or left neighborhood) $I \; \reflectbox{$\in$} \; a$ such that $F$ is constant on $I$. Otherwise we call $a$ a critical value. $F$ is tame if it has a finite number of critical values.
\end{definition}

In the classical case of finite dimensional vector spaces, the regularity condition requires that maps $F(a) \xrightarrow{\phi} F(b)$ are isomorphisms for $a, b$ in a neighborhood of a regular value (see~\cite[Def. 4.3]{bubenik_categorification_2014}). However, for a strict rank (such as $dim$ or more generally $length$) this is equivalent to our condition $r(F(a)) = r(F(b)) = r(im(\phi))$ thanks to the following lemma:

\begin{lemma}
    Let $r$ be a strict rank function and $A \xrightarrow{\phi} B$ a morphism such as $r(A) = r(B) = r(im(\phi))$. Then $\phi$ is an isomorphism.
\end{lemma}
\begin{proof}
    We have a natural regular epimorphism $A \overset{\chi}\twoheadrightarrow im(\phi)$ and $r(A) = r(im(\phi))$, so $\chi$ is an isomorphism. Similarly we have a natural monomorphism $im(\phi) \xhookrightarrow{\psi} B$ and $r(im(\phi)) = r(B)$ so $\psi$ is an isomorphism. $\phi = \psi \circ \chi$ is therefore also an isomorphism.
\end{proof}

We will need one more lemma to be able to use persistence functions to compute multiplicity of cornerpoints.

\begin{lemma}
    \label{lm:decreasing}
    Let $p$ be a persistence function on a category $\mathbf C$. Then, given a diagram
\[
    A \to B \to C \to D
\]
in the category $\mathbf C$, the function:
\[
    p(B \to C) - p(A \to C) - p(B \to D) + p(A \to D)
\]
is weakly decreasing in $A$ and $C$ and weakly increasing in $B$ and $D$.
\end{lemma}
\begin{proof}
    Let us prove that it is weakly decreasing in $A$, i.e. that given a diagram $A \to A^\prime \to B \to C \to D$, the following inequality holds
    \begin{align*}
        & p(B \to C) - p(A \to C) - p(B \to D) +  p(A \to D) \ge\\
        &p(B \to C) - p(A^\prime \to C) - p(B \to D) + p(A^\prime \to D)
    \end{align*}

\noindent Or, equivalently:
\[
     - p(A \to C) + p(A \to D) \ge - p(A^\prime \to C) + p(A^\prime \to D)
\]
which is simply the second property of~\cref{def:persistence}.
\end{proof}
\begin{definition}\label{def:corners}
    Given $u < v \in \mathbb R \cup \{-\infty, +\infty\}$ we define the multiplicity of $u, v$ as the minimum of the following expression, over $I_u, I_v$ disjoint connected neighborhoods of $u$ and $v$ respectively:
    \[
        p_{F}(\sup(I_u), \inf(I_v))-p_{F}(\inf(I_u), \inf(I_v)) - p_{F}(\sup(I_u), \sup(I_v))+p_{F}(\inf(I_u), \sup(I_v))
    \]
    We denote this quantity by $\mu(u, v)$. Whenever $\mu(u, v) > 0$ we say $(u, v)$ is a cornerpoint. By convention in this definition we consider $p_F(u, v) = \min_{x, y} p_F(x, y)$ whenever either $u$ or $v$ is not finite.
\end{definition}
\begin{remark}
    By~\cref{lm:decreasing}, the quantity:
    \[
        p_{F}(\sup(I_u), \inf(I_v))-p_{F}(\inf(I_u), \inf(I_v)) - p_{F}(\sup(I_u), \sup(I_v))+p_{F}(\inf(I_u), \sup(I_v))
    \]
    is weakly increasing in both $I_u$ and $I_v$ (where the ordering on the intervals is given by inclusion), so in practice this minimum is achieved for $I_u$ and $I_v$ sufficiently small intervals around $u$ and $v$ respectively.
\end{remark}

\begin{remark}[Cornerpoints at infinity]
We identify the vertical line $\varrho$ of equation $u=k$ with the pair $(k, +\infty)$. \Cref{def:corners} allows one to define the multiplicity $\mu(\varrho)$ as the minimum of
\[
    p_{F}(\sup(I_k), v)-p_{F}(\inf(I_k), v).
\]
Whenever $\mu(\varrho)>0$, we say that $\varrho$ is a {\em cornerpoint at infinity}.
\end{remark}

\begin{definition}\label{def:persistence_diagram}
  The persistence diagram $\mathcal{D}F$ associated with the persistence function $p_F$ is the multiset of its cornerpoints, along with all the diagonal points $\{(u,u) | u\in\mathbb{R}_{\geq 0}\}$ with infinite (countable) multiplicity.
\end{definition}

It is easy to show that if $F$ is tame the persistence diagram has only a finite number of off-diagonal points. The following property is relevant when measuring distances between diagrams and will be key in the remainder of this section.

\begin{proposition}
    \label{prop:countmultiplicities}
    If $\alpha<\beta \le \gamma<\delta \in \mathbf R \cup \{+\infty\}$ are right-regular points, then sum of the multiplicities of the cornerpoints $(u, v)$ s. t. $\alpha < u \le \beta$ and $\gamma < v \le \delta$ is
\[
    p_F(\beta, \gamma)-p_F(\alpha, \gamma)-p_F(\beta, \delta)+p_F(\alpha, \delta)
\]
\end{proposition}
\begin{proof}
    By induction on the number of cornerpoints in the box.
\end{proof}

\subsection{Indecomposable persistence modules}
\label{sec:indecomposable}

Given a tame (with respect to a strict rank) $\Rl$-indexed diagram $F \in \textnormal{Obj}(\mathbf D ^\Rl)$, we can partition $\mathbb R$ into a finite number of non-empty intervals $C_1, \dots, C_n \subseteq \mathbb R$ such that $F(x\le y)$ is an isomorphism whenever $x, y$ lie in the same interval. The full subcategory of such $\Rl$-indexed diagrams is equivalent to the category of representations of the poset $(\{1,\dots,n\}, \le)$. Given a sequence of points $c_i \in C_i$, the equivalence of the two representation categories is induced by the pair of order-preserving maps:
\begin{align*}
    \begin{aligned}
        & \iota: (\{1,\dots, n \}, \le) \to \Rl \\
        & i \mapsto c_i
    \end{aligned}&&
    \mbox{ and }&&
    \begin{aligned}
        & \pi: \Rl \to (\{1,\dots, n \}, \le) \\
        & x \mapsto i \mbox{ such that } x \in C_i
    \end{aligned}
\end{align*}

If $\mathbf D$ is an Abelian category of finite length, then so is $\mathbf D^{(\{1,\dots,n\},\le)}$. Indeed, we can bound the length of any $F \in \mathbf D^{(\{1,\dots,n\},\le)}$ as follows:
\[
    length(F) \le \sum_{i = 1}^n length(F(i))
\]

By Krull-Schmidt theorem~\cite{atiyah_krull-schmidt_1956}, $F$ can then be decomposed as direct sum of indecomposable objects
\[
    F = \bigoplus_{k\in K} I_k
\]

Indecomposable objects in $\mathbf D^{(\{1,\dots,n\},\le)}$ have been characterized in the case $\mathbf D = \mathbf{FinVec}_\mathbb{K}$. Indeed, let $A_n$ be the quiver having as nodes the points $\{1,\dots, n\}$ and non-trivial edges $i \to i+1$ for $i \in \{1, \dots, n-1\}$. Then, $A_n$ has trivially the same representations of $(\{1,\dots,n\}, \le)$ and is one of the ADE Dynkin diagrams for which Gabriel's theorem~\cite{gabriel_unzerlegbare_1972} can characterize all indecomposable representations. See~\cite{oudot_persistence_2015} for a treatment of persistence homology that takes Gabriel's theorem and Krull-Schmidt theorem as starting points. We can not, unfortunately, use Gabriel's theorem as we wish to work with a more general $\mathbf D$, but we will provide an equivalent classification for the quiver $A_n$ and $\mathbf D$ semisimple (\cref{def:semisimple_cat}). To do so, we will need to generalize $\cite[Def.~4.1]{bubenik_categorification_2014}$.

\begin{definition}\cite[Def.~4.1]{bubenik_categorification_2014}
    Given a semisimple Abelian category $\mathbf D$, a simple object $S\in\textnormal{Obj}(\mathbf D)$ and an interval $I \subseteq \mathbb R$ we define the diagram $\chi_{I, S} \in \mathbf{D}^\Rl$ as:
\begin{align*}
    \chi_{I, S}(a) = \begin{cases}{}
        S&\text{if } a \in I\\
        0 &\text{otherwise}
    \end{cases}&&
    \mbox{ and }&&
    \chi_{I, S}(a \le b) = \begin{cases}{}
        \textnormal{Id}_S&\text{if } a, b \in I\\
        0 &\text{otherwise}
    \end{cases}
\end{align*}
    When working in $\mathbf D^{(\{1,\dots,n\},\le)}$ we abuse of the same notation and write:

\[
    \chi_{[b, d], S} = \underbrace{0 \to \dots \to 0}_{[1, b-1]} \to \underbrace{S \xrightarrow{Id} \dots \xrightarrow{Id} S}_{[b, d]} \to \underbrace{0 \to \dots \to 0}_{[d+1, n]}
\]

\noindent We say that a $F$ has finite type if $F = \bigoplus_{k \in K} \chi_{I_k, S_k}$
\end{definition}

\begin{proposition}
    \label{prop:indecomposablequiver}
    If $\mathbf D$ is semisimple, all indecomposable objects in $\mathbf D^{(\{1,\dots,n\},\le)}$ are isomorphic to an ``interval object'' of the form $\chi_{[b, d], S}$ where $S$ is a simple object.
\end{proposition}
\begin{proof}
We proceed by contradiction. Let us take the smallest $n \in \mathbb N$ for which this does not hold and an indecomposable $F \in \mathbf D ^ {(\{1,\dots,n\},\le)}$ not isomorphic to any $\chi_{[b, d], S}$. $F(1) \not\simeq 0$ and $F(n) \not\equiv 0$ as otherwise we could find a counter-example for $\mathbf D ^ {(\{1,\dots,n-1\},\le)}$. Similarly $\phi = F(n-1 \le n)$ cannot be an isomorphism, as otherwise we would have a counter-example in $\mathbf D ^ {(\{1,\dots,n-1\},\le)}$. $\phi$ must be epi, otherwise, we could write $F(n) = im(\phi) \oplus C$ with $C \not\simeq 0$ and $F$ would be the direct sum of:
\begin{align*}
    i \mapsto \begin{cases}
        im(\phi) &\text{ if } i = n \\
        F(i) &\text{ otherwise }
    \end{cases}&&
    \mbox{ and }&&
    i \mapsto \begin{cases}
          C &\text{ if } i = n \\
          0 &\text{ otherwise}
      \end{cases}
\end{align*}

\noindent So, necessarily $\phi$ is not monic, as in an Abelian category morphisms that are both monic and epic are isomorphisms. We can decompose each $F(i)$ starting from $i=1$ and proceeding recursively, by setting $F(i) = ker(F(i\leq n))\oplus C_i$, for every $i\in\{1,2,\dots, n-1\}$, where we can take $C_i$ such that $F(i-1 \le i)(C_{i-1}) \subseteq C_i$. We can then decompose $F$ as a direct sum of:

\begin{align*}
    i \mapsto \begin{cases}
        F(n) &\text{ if } i = n \\
        C_i &\text{ otherwise}
    \end{cases}&&
    \mbox{ and }&&
    i \mapsto \begin{cases}
        0 &\text{ if } i = n \\
        ker(F(i \le n)) &\text{ otherwise}
    \end{cases}
\end{align*}

By assumption $F(n) \not\simeq 0$ and $ker(F(n-1 \le n)) \not\simeq 0$ so this is a non-trivial decomposition which is absurd.
\end{proof}

\begin{theorem}
    \label{thm:finitetype}
    In a semisimple Abelian category equipped with the rank function $length$, a $\Rl$-indexed diagram $F$ is of finite type if and only if it is tame.
\end{theorem}
\begin{proof}
    If $F$ is of finite type, then all points that are not extrema of some of the intervals defining $F$ are regular, so there can only be finitely many critical values.
    Conversely, if $F$ is tame, then we can find some partition of the real line in nonempty intervals $C_1, \dots, C_n \subseteq \mathbb R$ (which we assume to be sorted, i.e. $c_i < c_j$ whenever $c_i \in C_i$, $c_j, \in C_j$ and $i < j$) such that $F(x\le y)$ is an isomorphism whenever $x, y$ lie in the same interval. We can then build $\tilde{F} \in \textnormal{Obj}(\mathbf D)^{A_n}$ by $\tilde{F}(i) = F(c_i)$ (where $c_i$ is some point in $C_i$). By~\cref{prop:indecomposablequiver} $\tilde{F} \simeq \chi_{[b, d], S}$ for some $b, d \in \{1,\dots,n\}$ and some simple object $S$, so $F \simeq \chi_{I, S}$ where $I = \cup_{i = b}^d C_i$.
\end{proof}

\subsection{Interleaving and bottleneck distances}
\label{sec:bottleneckinterleaving}

There is a natural notion of distance between $\Rl$-indexed diagram, the interleaving distance. Here we recall the categorical notion of interleaving from~\cite{bubenik_categorification_2014}, which in turn draws from~\cite{chazal_proximity_2009}. Note that here we will only consider strong interleavings, thus not considering the weaker definition provided in~\cite{chazal_proximity_2009}.

As in~\cite{bubenik_categorification_2014} we define the translation functor $T_b: \Rl \to \Rl$ as $T_b(a) = a+b$ and the natural transformation $\eta_b: Id_\Rl \to T_b$ given by $\eta_b : a\le a+b$.

Given a $\Rl$-indexed diagram $F$, $FT_\epsilon$ is simply defined by $x\mapsto F(x+\epsilon)$. We will often compose functors to the left of natural transformation (thus applying the functor to the morphism the natural transformation returns) or to the right (thus calling the natural transformation on the object returned by the functor). For example, starting from $\eta_b: Id \to T_b$, we can compose $F$ to the left and obtain a new natural transformation $F\eta_b: F \to FT_b$. Then, similarly, we can compose $T_c$ to the right and obtain $F\eta_bT_c: FT_c \to FT_{b+c}$.

\begin{definition}\cite[Def. 3.4]{bubenik_categorification_2014}
We remind that given two $\Rl$-indexed diagrams $F, G$, they are $\epsilon$-interleaved if there are natural transformations $\phi^F: F \to GT_\epsilon$ and $\phi^G: G \to FT_\epsilon$ such that:
\begin{align*}
    (\phi^G T_\epsilon)\phi^F = F\eta_{2\epsilon}
    &&
    \mbox{ and }
    &&
    (\phi^F T_\epsilon)\phi^G = G\eta_{2\epsilon}
\end{align*}
The interleaving distance $d(F, G)$ is the infimum of all $\epsilon$ values such that $F$ and $G$ are $\epsilon$-interleaved.
\end{definition}

There is a simple example coming from filtering functions. Given a topological spaces $X$ and two real-valued functions $f, g: X \to \mathbb R$, if $f$ and $g$ differ no more than $\epsilon$, i.e., for all $x \in X$, $|f(x) - g(x)| \le \epsilon$, then there is a natural $\epsilon$-interleaving between the two $\Rl$-indexed diagrams corresponding to the sublevels of $f$ and $g$ respectively.

It is natural to define persistence starting from one or sometimes more functors (see~\cref{sec:poset} for an example with two functors):
\[
    \mathbf C_0 \xrightarrow{\Psi_1} \mathbf C_1 \xrightarrow{\Psi_2} \dots \xrightarrow{\Psi_n} \mathbf C_n
\]
where $\mathbf C_0, \dots, \mathbf C_{n-1}$ are arbitrary categories, whereas $(\mathbf C_n, r)$ is a ranked category. A $\Rl$-indexed diagram $F \in \mathbf C_0^\Rl$ is mapped by the various functors $\Psi_i$ in $\Rl$-indexed diagrams
\begin{equation*}
\Psi_1(F)\in \mathbf C_1^\Rl, \dots, \Psi_n(F)\in \mathbf C_n^\Rl
\end{equation*}
Similarly an $\epsilon$-interleaving between $F, G \in \mathbf C_0^\Rl$ is mapped to $\epsilon$-interleavings between $\Psi_1(F), \Psi_1(G)$, $\Psi_2(F), \Psi_2(G)$, et cetera. As a consequence we can define a sequence of interleaving distances $d_0 \ge d_1 \ge \dots \ge d_n$ as follows:

\[
    d_i(F, G) = d_{\mathbf C_i}(\Psi_i(F), \Psi_i(G))
\]
where $d_{\mathbf C_i}$ is the interleaving distance in category $\mathbf C_i$.

Furthermore, the bottleneck distance neglects the underlying category and is defined only via the persistence diagram.

\begin{definition}\label{def:bottleneck}
Let $F, G$ be two tame $\Rl$-indexed diagrams in $\mathbf R$ and ${\cal D}F, {\cal D}G$ their persistence diagrams. The bottleneck distance between the persistence diagrams is defined as
\[
d({\cal D}F, {\cal D}G) = \inf_{\beta\in{\cal B}} \sup_{p\in {\cal D}(F)} \| p-\beta(p)\|_\infty,
\]
where ${\cal B}$ is the collection of all bijections between ${\cal D}F$ and ${\cal D}G$.
\end{definition}

We now prove that under mild hypotheses ($\mathbf C_n$ admits finite colimits) the chain of decreasing distances can be continued to include the bottleneck distance
\[
    d_0(F, G) \ge d_1(F, G) \ge \dots \ge d_n(F, G) \ge d({\cal D}F, {\cal D}G)
\]
and find examples of ranked categories that achieve the equality $d_n(F, G) = d({\cal D}F, {\cal D}G)$. In particular, this chain of inequalities grants stability in the classical sense: as noted in~\cref{rem:indexeddiagram}, given two filtering functions that differ less than $\epsilon$, the associated $\Rl$-indexed diagrams are $\epsilon$-interleaved.

To prove inequalities between interleaving and bottleneck distance, we will generalize~\cite[Lm. 4.5]{chazal_proximity_2009} to the case of persistence function on an arbitrary category and~\cite[Lm. 4.6, 4.7]{chazal_proximity_2009} from the category of vector spaces to an arbitrary category with finite colimits.

\begin{lemma}[Box lemma]
    \label{lm:boxlemma}
    Let $F, G$ be two tame $\Rl$-indexed diagrams that are $\epsilon$-interleaved. Given $\alpha < \beta < \gamma < \delta$ let $\square$ denote the region $(\alpha, \beta] \times (\gamma, \delta]$ and $\square_\epsilon$ the region $(\alpha-\epsilon, \beta+\epsilon] \times (\gamma - \epsilon, \delta + \epsilon]$. Then the sum of the multiplicities of the points of ${\cal D}F$ contained in $\square$ is smaller or equal to the sum of the multiplicities of the points of ${\cal D}G$ contained in $\square_\epsilon$.
\end{lemma}
\begin{proof}
    As in~\cite[Lm. 4.5]{chazal_proximity_2009} we notice that, if $\beta+\epsilon > \gamma - \epsilon$, then $\square_\epsilon$ intersects the diagonal and so the total multiplicity of ${\cal D}G$ intersected with the diagonal is $\infty$, so we can assume $\beta+\epsilon \le \gamma - \epsilon$.

    As $F$ and $G$ are $\epsilon$-interleaved, we have the commutative diagram:

\[
\begin{tikzcd}[column sep=small]
    & F(\alpha) \arrow{r} & F(\beta) \arrow{rd} \arrow{rrr} & & & F(\gamma) \arrow{r} & F(\delta) \arrow{rd} &\\
    G(\alpha-\epsilon) \arrow{ru} \arrow{rrr} & & & G(\beta+\epsilon) \arrow{r} & G(\gamma-\epsilon) \arrow{ru} \arrow{rrr} & & & G(\delta+\epsilon)
\end{tikzcd}
\]

from which we can consider the sequence of morphisms

\[
    G(\alpha-\epsilon) \to F(\alpha) \to F(\beta) \to G(\beta+\epsilon) \to G(\gamma-\epsilon) \to F(\gamma) \to F(\delta) \to G(\delta+\epsilon)
\]

Let us first assume that $\alpha, \beta, \gamma, \delta$ are all right-regular values for $F$ and $\alpha - \epsilon, \beta + \epsilon, \gamma - \epsilon, \delta + \epsilon$ are all right-regular values for $G$. Then by~\cref{prop:countmultiplicities}, we can compute the sum of the multiplicities of the points of ${\cal D}F$ and ${\cal D}G$ using the categorical persistence function $p$ at the corners of the respective regions. Therefore we simply need to prove:
\[
    \begin{aligned}
    &p(G(\beta + \epsilon \le \gamma - \epsilon)) - p(G(\alpha-\epsilon \le \gamma - \epsilon)) - p(G(\beta + \epsilon \le \delta + \epsilon)) + p(G(\alpha-\epsilon \le \delta + \epsilon)) \ge \\
    &p(F(\beta \le \gamma )) - p(F(\alpha \le \gamma )) - p(F(\beta \le \delta )) + p(F(\alpha \le \delta ))
    \end{aligned}
\]

The inequality can be proven by repeatedly applying~\cref{lm:decreasing}. A smaller diagram, not including $F(\delta)$ and $G(\delta+\epsilon)$, can be used to prove the case $\delta = +\infty$.

If some of $\alpha, \beta, \gamma, \delta$ is not right-regular for $F$ or some of $\alpha - \epsilon, \beta + \epsilon, \gamma - \epsilon, \delta + \epsilon$ is not right-regular for $G$, we can simply prove the inequality for $\alpha^\prime, \beta^\prime, \gamma^\prime, \delta^\prime = \alpha + h, \beta + h, \gamma + h, \delta + h$, where $h$ is such that $\alpha^\prime, \beta^\prime, \gamma^\prime, \delta^\prime$ are right-regular points for $F$ and $\alpha^\prime - \epsilon, \beta^\prime + \epsilon, \gamma^\prime - \epsilon, \delta^\prime + \epsilon$ are right-regular for $G$. Taking the limit for $h \to 0^+$ ends the proof.
\end{proof}

\begin{lemma}[Interpolation lemma]
    \label{lm:interpolate}
    Let $\mathbf C$ be a category with finite colimits. If $F, G \in \mathbf C^\Rl$ are $\epsilon$-interleaved, there exists an interpolation $\tilde{H}_s$ for all $s \in [0, \epsilon]$ such that: $F$ and $\tilde{H}_s$ are $s$-interleaved, $G$ and $\tilde{H}_s$ are $(\epsilon-s)$-interleaved, $\tilde{H}_s$ and $\tilde{H}_{s^\prime}$ are $|s-s^\prime|$-interleaved.
\end{lemma}
\begin{proof}
    The proof follows the construction of~\cite{chazal_proximity_2009}, but in the more general setting of categories with finite colimits. We start by  defining $\e_1 = s$ and $\e_2 = \e-s$. Then $H_s = FT_{-\e_1}\amalg GT_{-\e_2}$. We have a natural transformation
    \[
      F \xrightarrow{\iota^F} H_sT_{\e_1} = F \amalg GT_{\e_1-\e_2}
    \]
    given by the coproduct inclusion as well as a natural transformation
    \[
     H_s = FT_{-\e_1}\amalg GT_{-\e_2} \xrightarrow{\pi^F} FT_{\e_1}
    \]
    which is defined as $F\eta_{2\e_1}T_{-\e_1}$ on the first term of the coproduct and as $\phi^G T_{-\e_2}$ on the second term of the coproduct.

    For this to be an interleaving, we need to prove that $(\pi^F T_{\e_1})\iota^F = F\eta_{2\e_1}$ (going from $F$ to $H_sT_{\e_1}$ and then to $FT_{2e\_1}$ versus going from $F$ to $FT_{2\e_1}$ directly) and that $\iota^F T_{\e_1} \pi^F = H_s\eta_{2\e_1}$ (going from $H_s$ to $FT_{\e_1}$ and then to $H_sT_{2\e_1}$ versus going from $H_s$ to $H_sT_{2\e_1}$ directly).

    We have $(\pi^F T_{\e_1})\iota^F = F\eta_{2\e_1}$ as the left hand side is the composition:

\[
    F \to F \amalg GT_{\e_1-\e_2} \to FT_{2\e_1}
\]

\noindent where the first morphism is the coproduct inclusion and the second morphism is $F\eta_{2\e_1}$ on the first component of the coproduct.

As remarked by~\cite[Appendix A]{chazal_proximity_2009}, however, $l^F = \iota^F T_{\e_1} \pi^F$ and $d^F = H_s\eta_{2\e_1}$ are not equal in general. Similarly, the $d^G$ and $l^G$ morphisms defined symmetrically are also not equal in general. $\tilde{H}_s$ is defined by coequalizing both $d^FT_{-2\e_1}$ with $l^FT_{-2\e_1}$ and $d^GT_{-2\e_2}$ with $l^GT_{-2\e_2}$. This would of course satisfy all the desired interleaving properties between $F$, $G$ and $H_s$ but we need to show that the existing natural transformations $H_s \to FT_{\e_1}$ and $H_s \to GT_{\e_2}$ pass to the coequalizer (i.e. induce natural transformations $\tilde{H}_s \to FT_{\e_1}$ and $\tilde{H}_s \to GT_{\e_2}$). As everything is symmetric, we only need to prove it for the map $H_s \to FT_{\e_1}$.

We start by proving that the transformation $H_s \to FT_{\e_1}$ passes to the coequalizer of $d^FT_{-2\e_1}$ and $l^FT_{-2\e_1}$. We observe that

\[
    H_sT_{-2\e_1} \to FT_{-\e_1} \to H_s \to FT_{\e_1}
\]

\noindent is the same as the more direct map

\[
    H_sT_{-2\e_1} \to H_s \to FT_{\e_1}
\]

\noindent as both the blue parallelogram and the green rightmost triangle are commutative in the following diagram:

\[
\begin{tikzcd}[column sep=small]
    & F(x-\e_1) \arrow{rd} \arrow[rd, dash, green, shift right=-.75ex]
                \arrow{rr}\arrow[rr, dash, blue, shift right=-.75ex]
                \arrow[rr, dash, green, shift right=.75ex]
    &
    &
    F(x+\e_1)
    \\
    H_s(x-2\e_1) \arrow{ru}\arrow[ru, dash, blue, shift right=-.75ex]
                 \arrow{rr}\arrow[rr, dash, blue, shift right=.75ex]
    &
    &
    H_s(x) \arrow{ru}
           \arrow[ru, dash, blue, shift right=.75ex]
           \arrow[ru, dash, green, shift right=-.75ex]
\end{tikzcd}
\]

\noindent therefore

\begin{align*}
    H_sT_{-2\e_1} \to FT_{-\e_1} \to H_s \to FT_{\e_1} & = H_sT_{-2\e_1} \to FT_{-\e_1} \to FT_{\e_1}\\
                                                       & = H_sT_{-2\e_1} \to H_s \to FT_{\e_1}
\end{align*}

\noindent Proving that the transformation $H_s \to FT_{\e_1}$ passes to the coequalizer of $d^GT_{-2\e_2}$ and $l^GT_{-2\e_2}$ is slightly trickier. We need to prove that:

\[
    H_sT_{-2\e_2} \to GT_{-\e_2} \to H_s \to FT_{\e_1} = H_sT_{-2\e_2} \to H_s \to FT_{\e_1}
\]

\noindent As $H_sT_{-2\e_2} = FT_{-\e_1-2\e_2}\amalg GT_{-3\e_2}$ we can prove the above equality on the two components separately. We consider the diagram:
\[
\begin{tikzcd}[column sep=small]
    F(x-\e_1-2\e_2) \arrow{rd} \arrow[rd, dash, red, shift right=.75ex]
                               \arrow[rd, dash, teal, shift right=-.75ex]
                    \arrow{rrrr} \arrow[rrrr, dash, red, shift right=-.75ex]
                                 \arrow[rrrr, dash, teal, shift right=.75ex]
    & & & &
    F(x+\e_1) \\
    &
    H_s(x-2\e_2) \arrow{rd} \arrow[rd, dash, green, shift right=.75ex]
                            \arrow[rd, dash, teal, shift right=-.75ex]
                 \arrow{rr} \arrow[rr, dash, blue, shift right=-.75ex]
                            \arrow[rr, dash, red, shift right=.75ex]
    & &
    H_s(x) \arrow{ru} \arrow[ru, dash, red, shift right=.75ex]
                      \arrow[ru, dash, teal, shift right=-.75ex]
    \\
    G(x -3\e_2) \arrow{rr}\arrow[rr, dash, blue, shift right=.75ex]
                          \arrow[rr, dash, green, shift right=-.75ex]
                \arrow{ru}\arrow[ru, dash, blue, shift right=-.75ex]
                          \arrow[ru, dash, green, shift right=.75ex]
    & &
    G(x-\e_2) \arrow{ru}\arrow[ru, dash, blue, shift right=.75ex]
                        \arrow[ru, dash, teal, shift right=-.75ex]
\end{tikzcd}
\]
As the blue bottom parallelogram and the green bottom-left triangle are commutative, we have:
\[
    \begin{aligned}
    GT_{-3\e_2} \to H_sT_{-2\e_2} \to GT_{-\e_2} \to H_s \to FT_{\e_1} & = GT_{-3\e_2} \to GT_{-\e_2} \to H_s \to FT_{\e_1} \\
                                                                   & = GT_{-3\e_2} \to H_sT_{-2\e_2} \to H_s \to FT_{\e_1}
\end{aligned}
\]
As a consequence of the interleaving between $F$ and $G$, the large inverted teal triangle is also commutative and so is the top red trapezoid. Consequently, we have:
\[
    \begin{aligned}
    FT_{-\e_1-2\e_2} \to H_sT_{-2\e_2} \to GT_{-\e_2} \to H_s \to FT_{\e_1} & = FT_{-\e_1-2\e_2} \to GT_{-\e_2} \to FT_{\e_1} \\
                                                                        & = FT_{-\e_1-2\e_2} \to FT_{\e_1} \\
                                                                        & = FT_{-\e_1-2\e_2} \to H_sT_{-2\e_2} \to H_s \to FT_{\e_1}
\end{aligned}
\]
so, necessarily
\[
    H_sT_{-2\e_2} \to GT_{-\e_2} \to H_s \to FT_{\e_1} = H_sT_{-2\e_2} \to H_s \to FT_{\e_1}
\]
Proving that morphisms of the type $H_s \to H_{s^\prime}T_{|s-s^\prime|}$ also pass to the coequalizer is a similar exercise in diagram chasing.
\end{proof}

The following result is a generalization, in our setting, of~\cite[Thm. 4.4]{chazal_proximity_2009}. Given~\cref{lm:boxlemma,lm:interpolate}, which are the equivalent of~\cite[Lm. 4.5, 4.6, 4.7]{chazal_proximity_2009}, the proof of the following result is identical to the proof of~\cite[Thm. 4.4]{chazal_proximity_2009}: we reproduce it here with slight changes to adjust for differences in notation.

\begin{theorem}
    \label{thm:smallbottleneck}
    Let $\mathbf R$ be a category with finite colimits and $p$ be a categorical persistence function. If $F1, F2$ are two tame $\Rl$-indexed diagrams in $\mathbf R$, then $$d_{\mathbf R}(F1, F2) \ge d({\cal D}F1, {\cal D}F2).$$
\end{theorem}
\begin{proof}
    The proof is analogous to~\cite[Thm. 4.4]{chazal_proximity_2009}. Let us assume that $F, G$ are $\epsilon$-interleaved. We can construct $\tilde{H}_s$ as in~\cref{lm:interpolate}. We define:
    \[
        \delta(s) = \frac{1}{2}\min\left\{ || p-q ||_\infty,\; p \in {\cal D}\tilde{H}_s \setminus \Delta,\; q \in {\cal D}\tilde{H}_s \setminus \{p\} \right\}
    \]
    We say that $\tilde{H}_{s^\prime}$ is very close to $\tilde{H}_{s}$ if $|s-s^\prime|<\delta(s)$. In such case, by~\cref{lm:boxlemma}, as $\tilde H_s$ and $\tilde H_{s^\prime}$ are $|s-s^\prime|$ interleaved, any off-diagonal point of ${\cal D}\tilde H_s$ admits exactly one point of ${\cal D}\tilde H_{s^\prime}$ within $l^\infty$ distance $|s-s^\prime|$. By compactness, we can find a sequence $0 = s_0 < s_1 < \dots < s_n < s_{n+1} = \epsilon$ such that for $i = 0, \dots, n$ either $\tilde{H}_{s_i}$ is very close to $\tilde{H}_{s_{i+1}}$ or vice versa. From the Easy Bijection Lemma~\cite{cohen-steiner_stability_2007}, it follows that $d({\cal D}\tilde H_{s_i}, {\cal D}\tilde H_{s_{i+1}}) \le s_{i+1}-s_i$. By applying repeatedly the triangle inequality we obtain that $d({\cal D}F, {\cal D}G) \le \epsilon$.
\end{proof}

 Even though the interleaving distance is, under mild assumptions, larger than the bottleneck distance, the opposite is not true with such generality. In the rest of this section we will show a class of categories and rank functions for which the converse holds.

\begin{definition}
    \label{def:tightrank}
    A ranked category $(\Rcat, r)$ is tight if, for any tame $\Rl$-indexed diagrams $F$ and $G$, the following equality holds:
\[
    d_{\mathbf R}(F, G) = d({\cal D}F, {\cal D}G)
\]
\end{definition}

\begin{theorem}
    \label{thm:onesimple}
    Let $\mathbf D$ be a semisimple Abelian category with only one simple object up to isomorphism, equipped with the rank $length$. Then the interleaving and bottleneck distances coincide on tame $\Rl$-indexed diagrams, that is to say $(\mathbf D, length)$ is tight.
\end{theorem}
\begin{proof}
    Given $F, G$ two $\Rl$-indexed diagrams, we know already $d_{\mathbf D}(F, G) \le d({\cal D}F, {\cal D}G)$ because of~\cref{thm:smallbottleneck}. To prove the inequality, let us call $S$ the only (up to isomorphism) simple object in $\mathbf D$. As $F$ and $G$ are tame, by~\cref{thm:finitetype}, they are also of finite type and we can therefore write:
\begin{align*}
    F \simeq \bigoplus_{k \in {\cal D}F} \chi_{I_k, S}&&
    \mbox{ and }&&
    G \simeq \bigoplus_{k \in {\cal D}G} \chi_{I_k, S}
\end{align*}
where $S$ is a representative of the unique isomorphism class of simple objects in $\mathbf D$. We take $I_k$ to be the empty interval if $k$ lies on the diagonal of the persistence diagram.

Given $\epsilon > d({\cal D}F, {\cal D}G)$, let us take a bijection of persistence diagrams $\psi:{\cal D}F \to {\cal D}G$ which sends each point to a point of distance $<\epsilon$. The interleaving map $\phi^F: F \to GT_\epsilon$ will send $\chi_{I_k, S}$ into $\chi_{I_{\psi(k), S}} T_{\epsilon}$.
\end{proof}

\cref{thm:onesimple} is more general than the usual result (which considers $\mathbf D = \mathbf{FinVec}_\mathbb{K}$), as it includes modules over non-commutative division rings, which are a semisimple category with essentially one simple object. This will be important in the follow up to make general theorems about multicolored persistence in semisimple categories.

Having, up to isomorphism, only one simple object is a necessary assumption. As a counter-example, given a semisimple Abelian category $\mathbf D$ with at least two non-isomorphic simple objects $O_1$ and $O_2$, $(\mathbf D, length)$ as in~\cref{prop:length} is not tight. If we take two constant $\Rl$-indexed diagrams: $u \mapsto O_1$ and $u \mapsto O_2$, their interleaving distance is $\infty$ and their bottleneck distance is $0$. To recover the equality, we will use the concept of coloring.

\section{Multicolored persistence}
\label{sec:multicoloredpersistence}

The aim of this section is to find a suitable way to still use persistence diagrams to compute the interleaving distance (or find tighter bounds for it) even in categories with many non-isomorphic indecomposable objects. We will do so by defining \textit{multicolored persistence diagrams}, where each color encodes the isomorphism class of an indecomposable persistence module.

We start by introducing the concept of coloring of ranked categories.
\begin{definition}
    \label{def:colorablecategory}
    Given an index set $\Gamma$, we say that a ranked category $(\mathbf R, r)$ is $\Gamma$-colorable if there exist ranked categories $(\mathbf Q_\gamma, r_\gamma)$ for $\gamma \in \Gamma$ and image-preserving functors ${\cal C}_\gamma: \mathbf R \to \mathbf Q_\gamma$ such that
    \begin{enumerate}
        \item the induced functor ${\cal C} : \mathbf R \to \prod_{\gamma \in \Gamma}\mathbf Q_\gamma$ is fully faithful;
        \item for each $X \in \Obj(\Rcat)$, $r_\gamma(X)$ is $0$ for all but finitely many $\gamma \in \Gamma$ and:
            \begin{equation}
                \label{eq:multicoloredrank}
                r(X) = \sum_{\gamma \in \Gamma} r_\gamma({\cal C}_\gamma(X))
            \end{equation}
    \end{enumerate}
    We call such ${\cal C}$ a $\Gamma$-coloring and we say that $(\Rcat,r)$ is ${\cal C}$-colored.
\end{definition}

The fully faithful condition may sound quite abstract, in practice what we are asking is that given $X,Y \in \textnormal{Obj}(\Rcat)$ the natural map of sets:

\[
    \textnormal{Hom}(X, Y) \to \prod_{\gamma \in \Gamma} \textnormal{Hom}({\cal C}_\gamma(X), {\cal C}_\gamma(Y))
\]
is bijective.

Note that in an Abelian category $\mathbf C$ of finite length, equipped with the rank function $length$, we have a coloring given by the block decomposition into indecomposable categories $\mathbf C_\gamma$ (see~\cite[Sect.~1]{etingof_tensor_2015}). Consequently~\cref{eq:multicoloredrank} follows from the additivity of $length$.

\subsection{Multicolored persistence diagrams}
\label{sec:multicolored_persistence_diagrams}

In what follows, we will show how, given $(\mathbf R, r)$ a $\mathcal{C}$-colored ranked category, it is possible to construct a multicolored persistence diagram. First we will need a simple lemma:

\begin{lemma}
    Let $(\mathbf R, r)$ be a $\mathcal{C}$-colored ranked category. Given $F \in \mathbf R^\Rl$, if $F$ is constant on an interval $I$ with respect to $r$, then $F$ is also constant on $I$ with respect to all colored components $r_\gamma$. As a consequence, if $F$ is tame with respect to $r$, then $F$ is tame with respect to $r_\gamma$, for all $\gamma\in\Gamma$.
\end{lemma}
\begin{proof}
    Let $I\subseteq \R$ be an interval. Given $u\le v \in I$ we know that $r(F(u)) = r(im(F(u\le v)))$, i.e.:
    \[
        \sum_{\gamma \in \Gamma} r_\gamma(F(u)) = \sum_{\gamma \in \Gamma} r_\gamma(im(F(u \le v)))
    \]
    Of course, for any $\gamma$, we have $r_\gamma(F(u)) \ge r_\gamma(im(F(u \le v)))$. If for some $\overline\gamma$ we had a strict inequality $r_{\overline\gamma}(F(u)) > r_{\overline\gamma}(im(F(u \le v)))$, then we would have a strict inequality:

    \[
        \sum_{\gamma \in \Gamma} r_\gamma(F(u)) > \sum_{\gamma \in \Gamma} r_\gamma(im(F(u \le v)))
    \]
    which is absurd. $r_\gamma(im(F(u \le v))) = r_\gamma(F(v))$ is proved in the same way.
\end{proof}

As a consequence, given a $\cal C$-colored ranked category $(\Rcat, r)$ and a $\Rl$-indexed diagram $F$, we can draw its multicolored persistence diagram by superimposing the persistence diagrams associated to each colored component, see~\cref{fig:colored_graph_b}. Let $(\Rcat, r)$ be a $\cal C$-colored ranked category. The \textit{multicolored bottleneck distance} between two multicolored persistence diagrams is computed just like the normal bottleneck distance, but only accepting bijections that preserve the color of cornerpoints. We denote it by $d_{\cal C}$.

\begin{definition}\label{def:tight_coloring}
    Let $\cal C$ be a coloring on a ranked category $(\Rcat, r)$. We say that $\cal C$ is tight if for any tame $\Rl$-indexed diagrams $F, G$ the following equality holds:
\[
    d_{\mathbf R}(F, G) = d_{\cal C}({\cal D}F, {\cal D}G)
\]
\end{definition}

The multicolored bottleneck distance is greater or equal than the normal bottleneck distance, as the minimum is calculated across a smaller set of possible bijections. First, it is natural to ask whether the multicolored bottleneck distance is still bounded by the interleaving distance.

\begin{figure}[tb]
\centering
\subfigure[][]{\includegraphics[width=.4\textwidth]{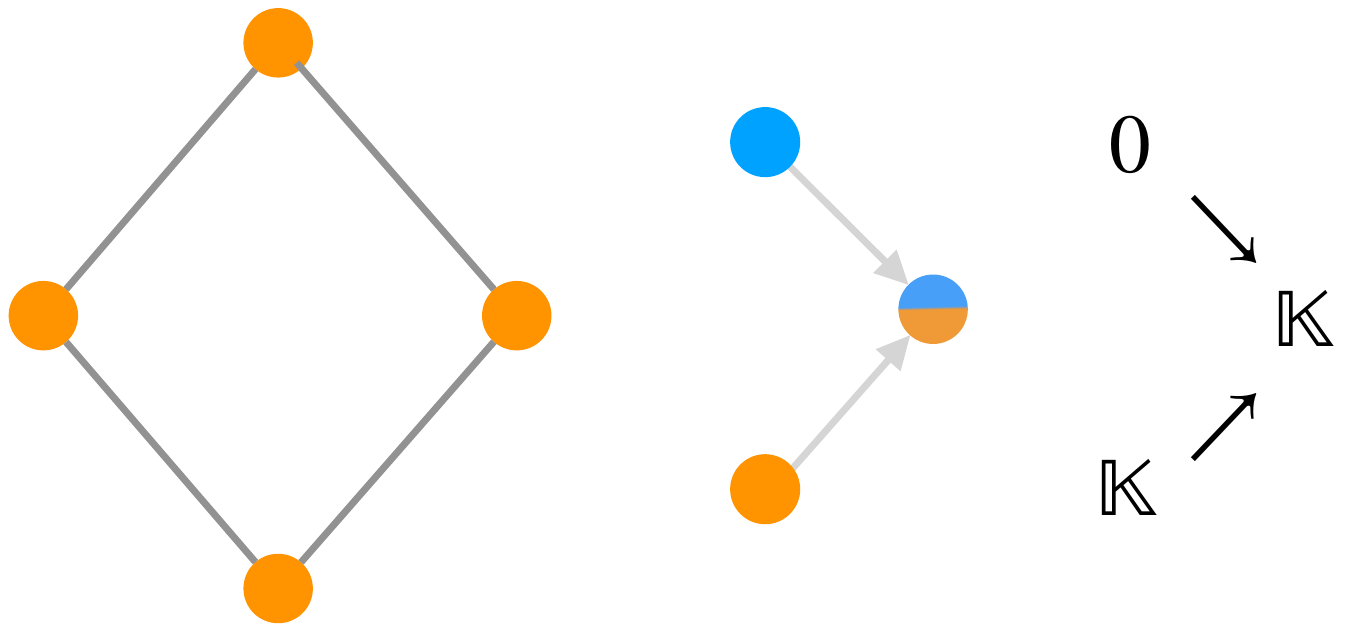}\label{fig:col_toy_a}}
\qquad
\subfigure[][]{\includegraphics[width=.4\textwidth]{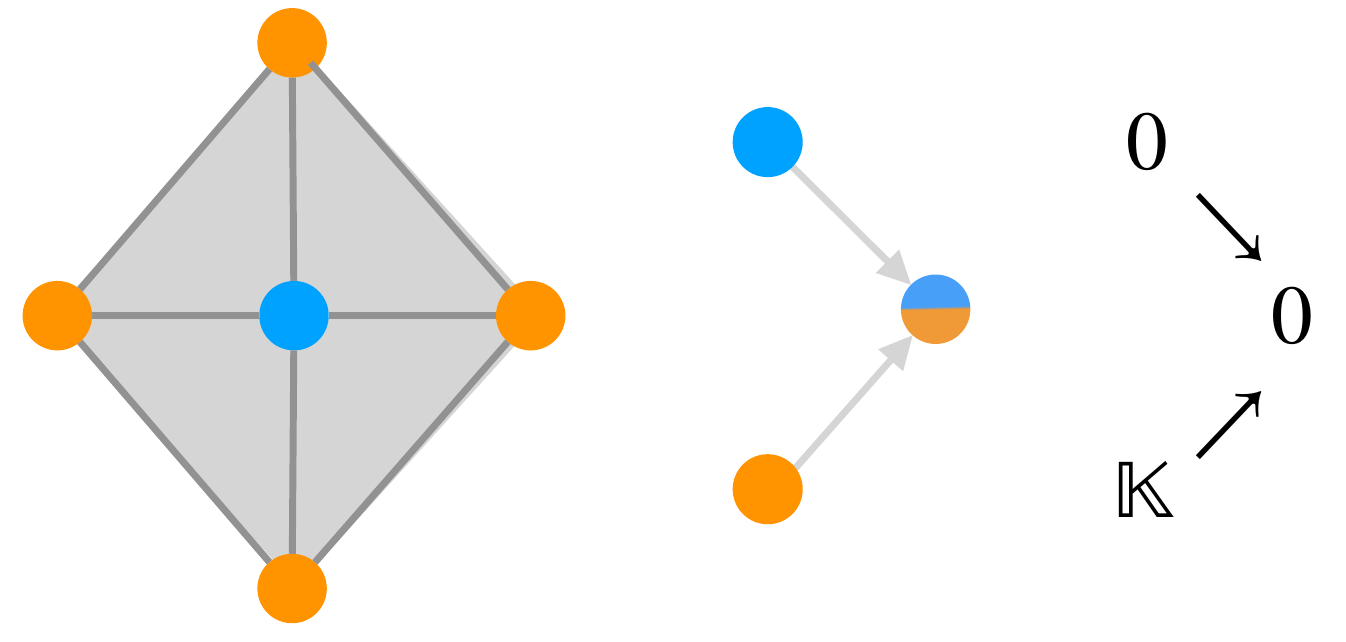}\label{fig:col_toy_b}}
\caption{\textbf{$H_1$ and multicolored simplicial complexes.} Coloring the vertices of a simplicial complex allows one to create a compact poset representation of all possible interaction between colored components evaluated through a rank function. Here, we consider a two-class coloring, namely orange $o$ and blue $b$. In each panel, the leftmost object is a multicolored simplicial complex, in the center the poset obtained by considering colors and their interactions: $(\left\{\{b\}, \{o\}, \{b, o\}\right\},\subseteq)$. Finally, the diagram obtained by computing the first homology group for each element of $\left\{\{b\}, \{o\}, \{b, o\}\right\}$. Observe how the cycle generated by the orange component propagates to $\{b, o\}$ in Panel (a), whilst it disappears from $\{b, o\}$ in Panel (b), because of the central blue vertex. }\label{fig:col_toy}
\end{figure}

\begin{theorem}
    \label{thm:tight}
    Let ${\cal C}: \mathbf R \to \prod_{\gamma\in\Gamma} \mathbf Q_\gamma$ be a $\Gamma$-coloring of a ranked category $(\Rcat, r)$ with components $(\mathbf Q_\gamma, r_\gamma)$. If, for all $\gamma \in \Gamma$, the category $\mathbf Q_\gamma$ admits finite colimits, then the multicolored bottleneck distance is bounded by the interleaving distance, i.e.
    \[
        d_{\mathbf R}(F, G) \ge d_{\cal C}({\cal D}F, {\cal D}G)
    \]

    \noindent Furthermore, if all $(\mathbf Q_\gamma, r_\gamma)$ are tight, then ${\cal C}$ is a tight coloring in the sense of~\cref{def:tight_coloring}.
\end{theorem}
\begin{proof}
    The functor ${\cal C}$ is fully faithful by~\cref{def:colorablecategory}, hence $F$ and $G$ are $\e$-interleaved in $\mathbf R$ if and only if ${\cal C}_\gamma F$ and ${\cal C}_\gamma G$ are  $\epsilon$-interleaved in $\mathbf Q_\gamma$ for all $\gamma$. Therefore:
    \[
        d_{\mathbf R}(F, G) = \sup_{\gamma\in\Gamma} \; d_{\mathbf Q_\gamma}({\cal C}_\gamma F, {\cal C}_\gamma G)
    \]
    Similarly:
    \[
        d_{\cal C}({\cal D}F, {\cal D}G) = \sup_{\gamma\in\Gamma} \; d({\cal D(C}_\gamma F), {\cal D(C}_\gamma G))
    \]
    As all $\mathbf Q_\gamma$ have finite colimits, thanks to~\cref{thm:smallbottleneck}, we have element-wise inequalities
    \[
        d_{\mathbf Q_\gamma}({\cal C}_\gamma F, {\cal C}_\gamma G) \ge d({\cal D(C}_\gamma F), {\cal D(C}_\gamma G))
    \]
    so, necessarily
    \[
     d_{\mathbf R}(F, G) \ge d_{\cal C}({\cal D}F, {\cal D}G)
    \]
    Similarly, if all $(\mathbf Q_\gamma, r_\gamma)$ are tight, we have element-wise equalities:
    \[
        d_{\mathbf Q_\gamma}({\cal C}_\gamma F, {\cal C}_\gamma G) = d({\cal D(C}_\gamma F), {\cal D(C}_\gamma G))
    \]
    and thus
    \[
     d_{\mathbf R}(F, G) = d_{\cal C}({\cal D}F, {\cal D}G)
    \]
\end{proof}

Given an Abelian semisimple category $\Cat$, let $\Gamma$ be a maximal set of non-isomorphic simple objects of $\Cat$ and, for each $\gamma \in \Gamma$, $\Cat_\gamma$ the full subcategory spanned by objects isomorphic to $\bigoplus_{i=1}^n \gamma$ for $n\in\mathbb N$. As in the case of finite group representations~\cite{serre_linear_2014}, objects in $\Cat$ can be canonically decomposed as a direct sum of components in the various $\Cat_\gamma$. This decomposition induces a natural tight coloring ${\cal C}:\mathbf C \to \prod_{\gamma\in\Gamma} \mathbf{C}_\gamma $ on $(\Cat, length)$.
\begin{theorem}
    Given $\mathbf C$ an Abelian semisimple category, ${\cal C}$ is a tight coloring on $(\mathbf C, length)$.
\end{theorem}
\begin{proof}
    By~\cref{thm:tight} we only need to prove that, for all $\gamma \in \Gamma$, the ranked category $(\Cat_\gamma, length)$ is tight. However the category $\mathbf C_\gamma$ is semisimple and has only one simple object $\gamma$ up to isomorphism so, by~\cref{thm:onesimple} it is tight.
\end{proof}

We have two examples in mind: groups and posets.

\subsection{Persistent homology on simplicial complexes with a group action}
\label{sec:group}

\begin{figure}[tb]
\centering
\subfigure[][]{\includegraphics[width=.3\textwidth]{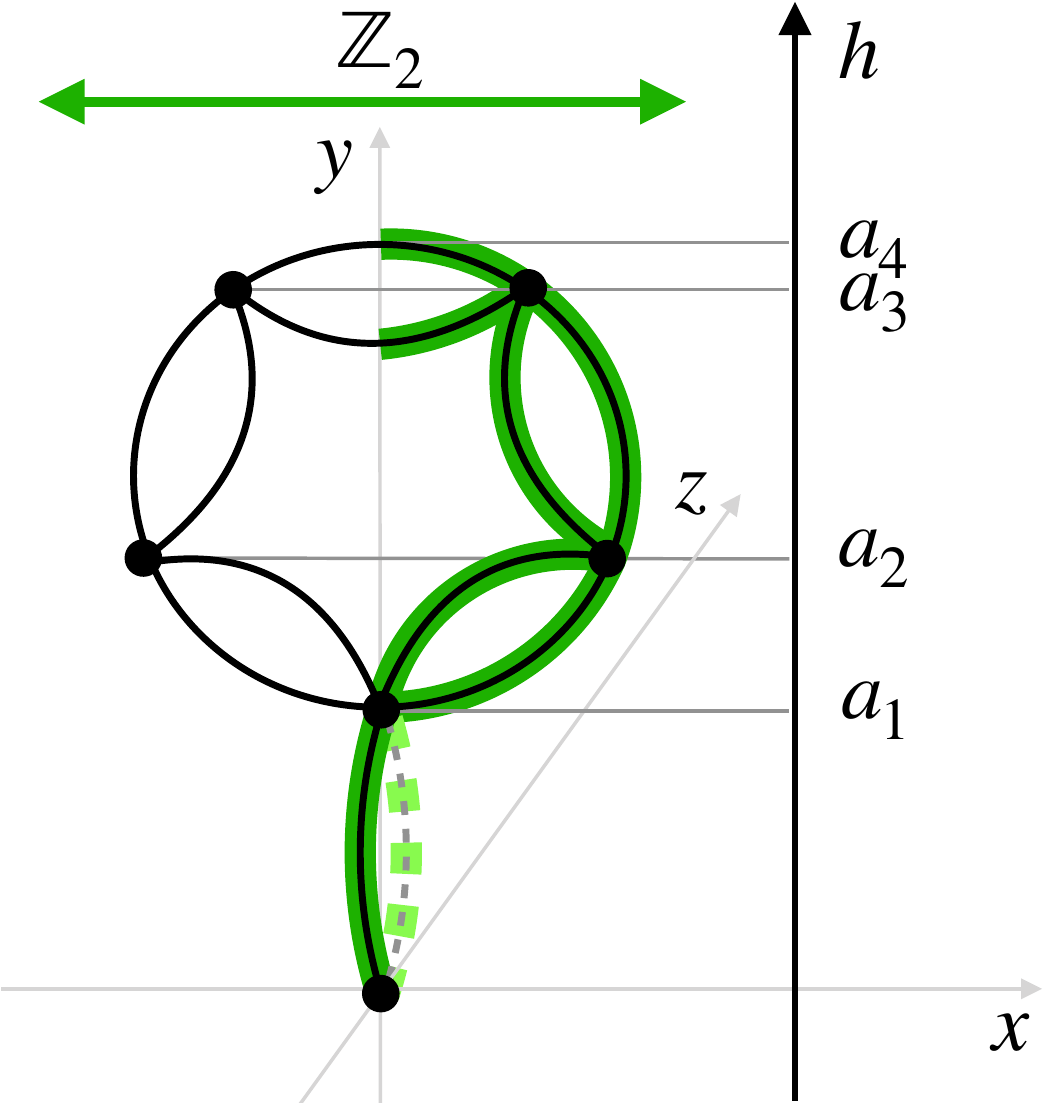}\label{fig:group_action_a}}
\qquad
\subfigure[][]{\includegraphics[width=.18\textwidth]{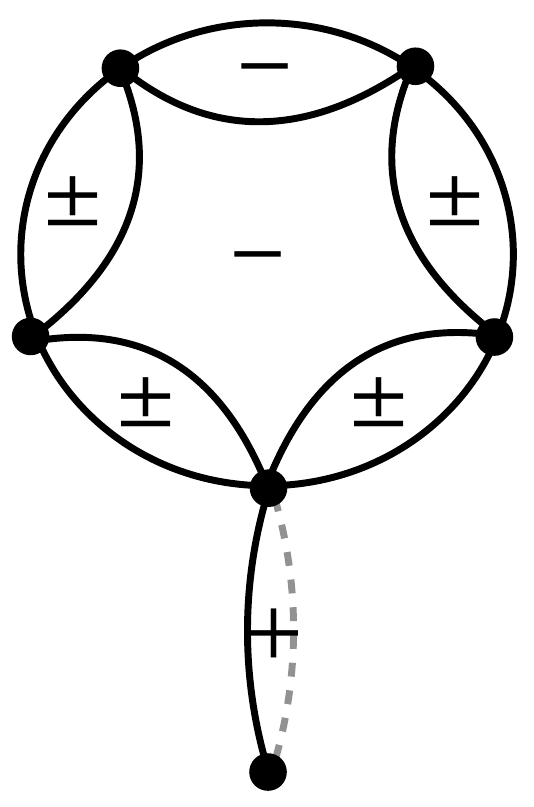}\label{fig:group_action_c}}
\qquad
\subfigure[][]{\includegraphics[width=.3\textwidth]{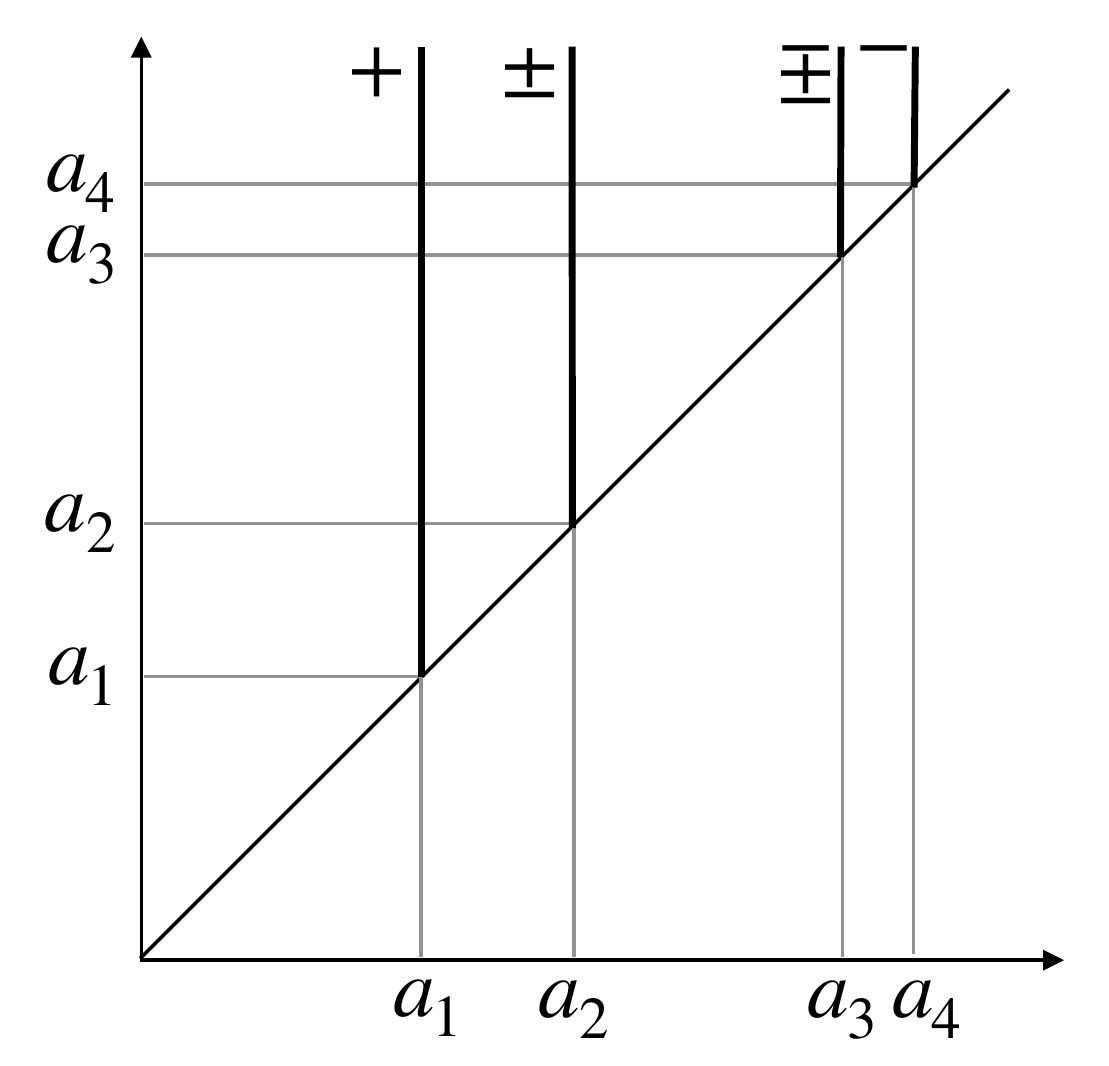}\label{fig:group_action_d}}
\caption{\textbf{$H_1$ multicolored persistence in $\mathbf{Top}^{\mathbb{Z}_2}$.} We consider the action of $\mathbb{Z}_2$ on the space $X$ represented in Panel (a) and generating the quotient highlighted in green. Note how the dashed loop lying on the $(y,z)$-plane is fixed by the action of group. We consider the filtration induced by the height function $h:X\rightarrow\mathbb{R}$. In Panel (b) cycles are labelled according to the group action. The same labeling is reported in the persistence diagram of Panel (c).}\label{fig:group_action}
\end{figure}

If $G$ is a finite group whose cardinality is not a multiple of the characteristic of $\mathbb K$, then the category of $G$-representations in $\mathbf{FinVec}_{\mathbb K}$ is semisimple.
The homology functor induces a map
\(
    H_n: \mathbf{FinSimp}^G \to \mathbf{FinVec}_{\mathbb K}^G
\)
which allows us to define a categorical persistence function on finite simplicial complexes with a $G$-action.

If a filtering function $f: X \to \mathbb R$ is $G$-invariant, i.e. for all $x \in X, g \in G$, $f(x) = f(gx)$, then the sublevels of $f$ form a $\Rl$-indexed diagram in $\mathbf{FinVec}_\mathbb{K}^G$. In practice it may well happen that the filtration and the group action are not compatible and the condition $f(x) = f(gx)$ is not always respected. In this case, one can consider an adjusted filtration such as:
\[
    \overline{f}(x) := \frac{1}{|G|}\sum_{g \in G} f(gx)
\]
\paragraph{Application: Vietoris-Rips filtration under a group action}
The above construction also applies on simiplicial complexes arising from point cloud data. Let $G$ be a finite group and $(X, d)$ a finite metric space with a distance-preserving $G$-action. The Vietoris-Rips filtration~\cite{carlsson2009topology} on $(X, d)$ is $G$-invariant and therefore induces a $\Rl$-indexed diagram in $\mathbf{FinSimp}^G$. Again, if the group action is not distance preserving, we can define an adjusted distance $\overline{d}(x, y) := \frac{1}{|G|}\sum_{g \in G} d(g\cdot x_i, g\cdot x_j)$.

\subsection{Persistent homology on labeled point clouds}
\label{sec:poset}

Let $(P, \preceq)$ be a finite poset. We can consider the category $\mathbf{FinSimp}^{(P, \preceq)}$, i.e. $(P, \preceq)$-indexed diagrams of finite simplicial complexes. We have a chain of functors:
\begin{equation*}
  \mathbf{FinSimp}^{(P, \preceq)} \xrightarrow{H_k} \mathbf{FinVec}_{\mathbb K}^{(P, \preceq)} \xrightarrow{Q} \mathbf{FinVec}_{\mathbb K}^{|P|}
\end{equation*}
where we define $Q$ as follows:
\[
    p \mapsto coker\left(\textstyle\bigoplus_{i \precneqq p} F(i) \to F(p)\right)
\]
 That is to say $Q(F)$ maps $p$ to $F(p)$ quotiented by the images of $F(i)$ with $i$ strictly smaller than $p$. As $\mathbf{FinVec}_{\mathbb K}^{|P|}$ is an Abelian semisimple category, the results of~\cref{sec:multicolored_persistence_diagrams} hold.

\begin{figure}[tb]
\centering
\subfigure[][]{\includegraphics[width=.35\textwidth]{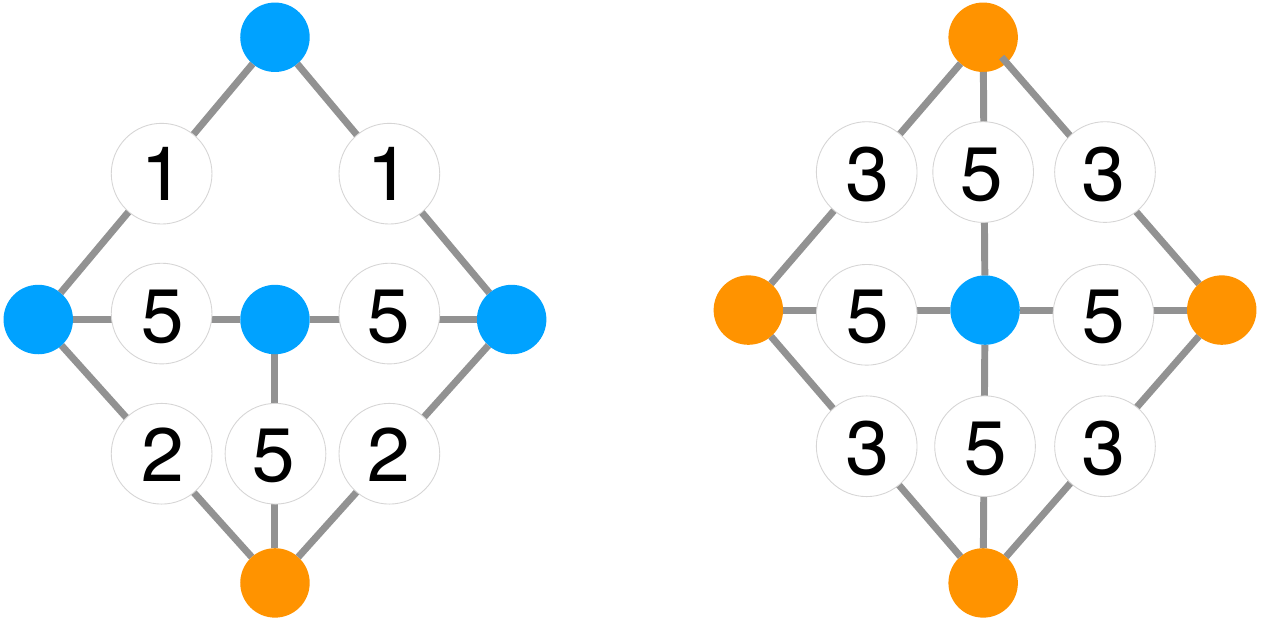}\label{fig:colored_graph_a}}
\qquad\qquad\qquad
\subfigure[][]{\includegraphics[width=.2\textwidth]{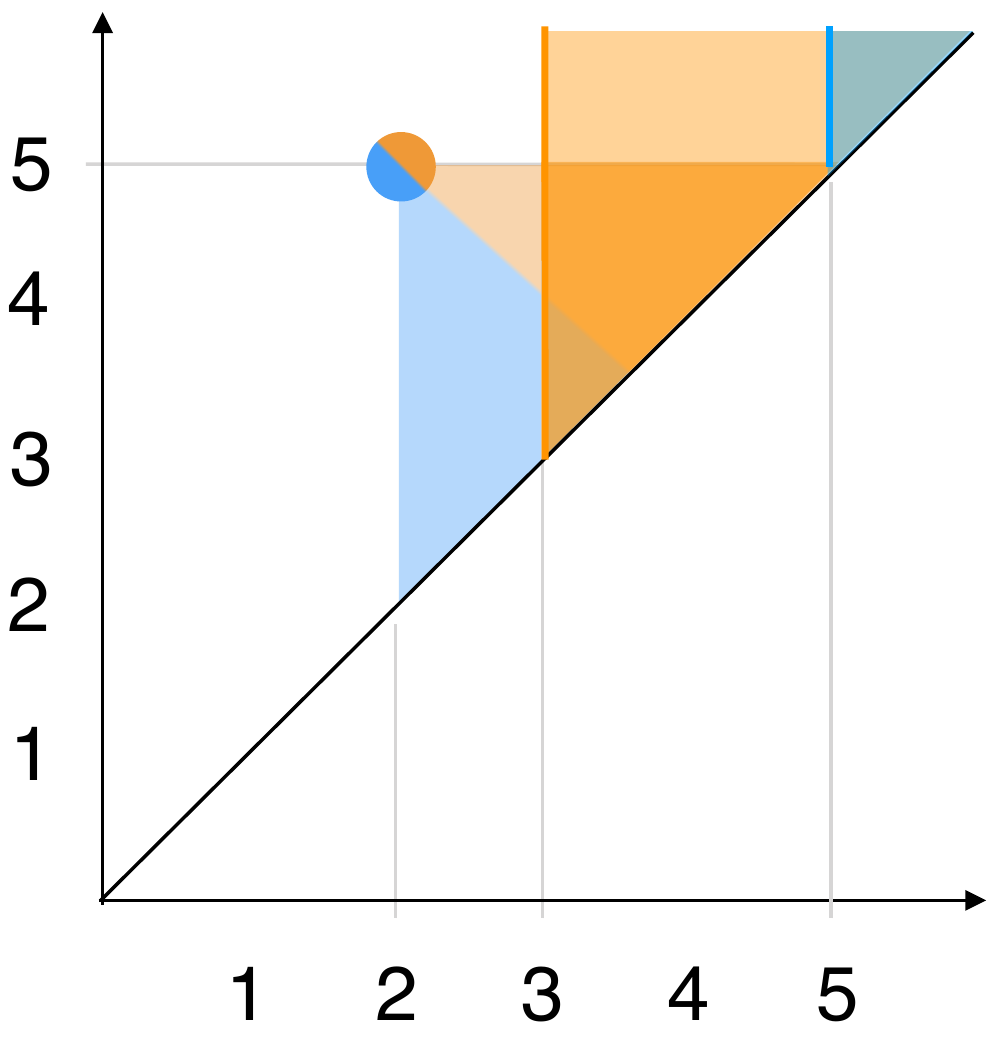}\label{fig:colored_graph_b}}

\subfigure[][]{\includegraphics[width=.8\textwidth]{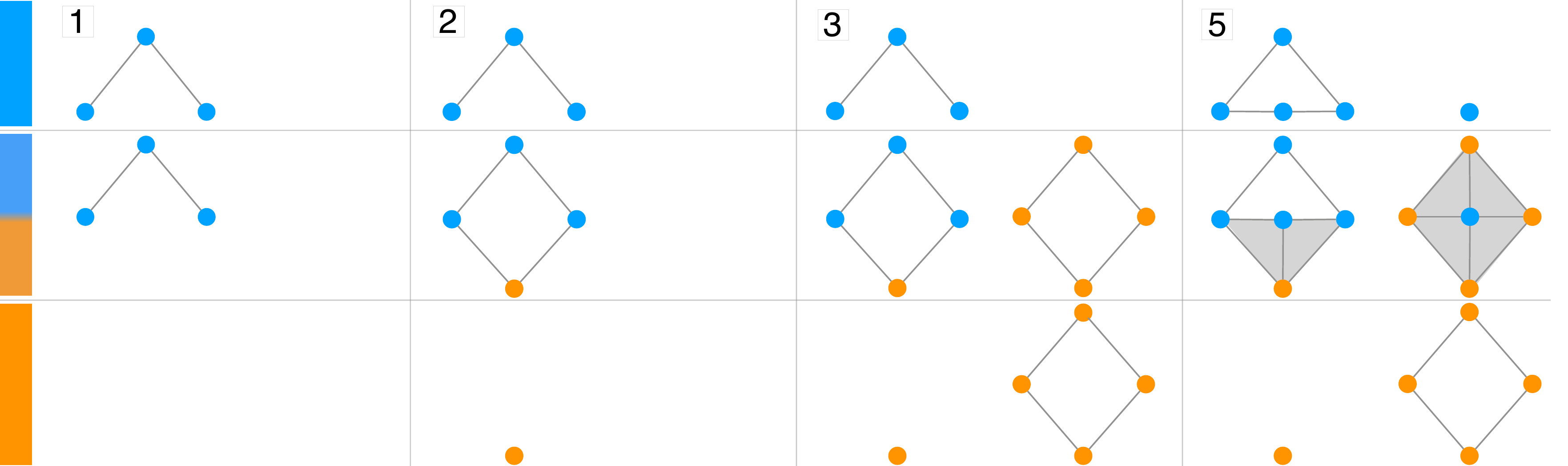}\label{fig:colored_poset_a}}

\subfigure[][]{\includegraphics[width=.4\textwidth]{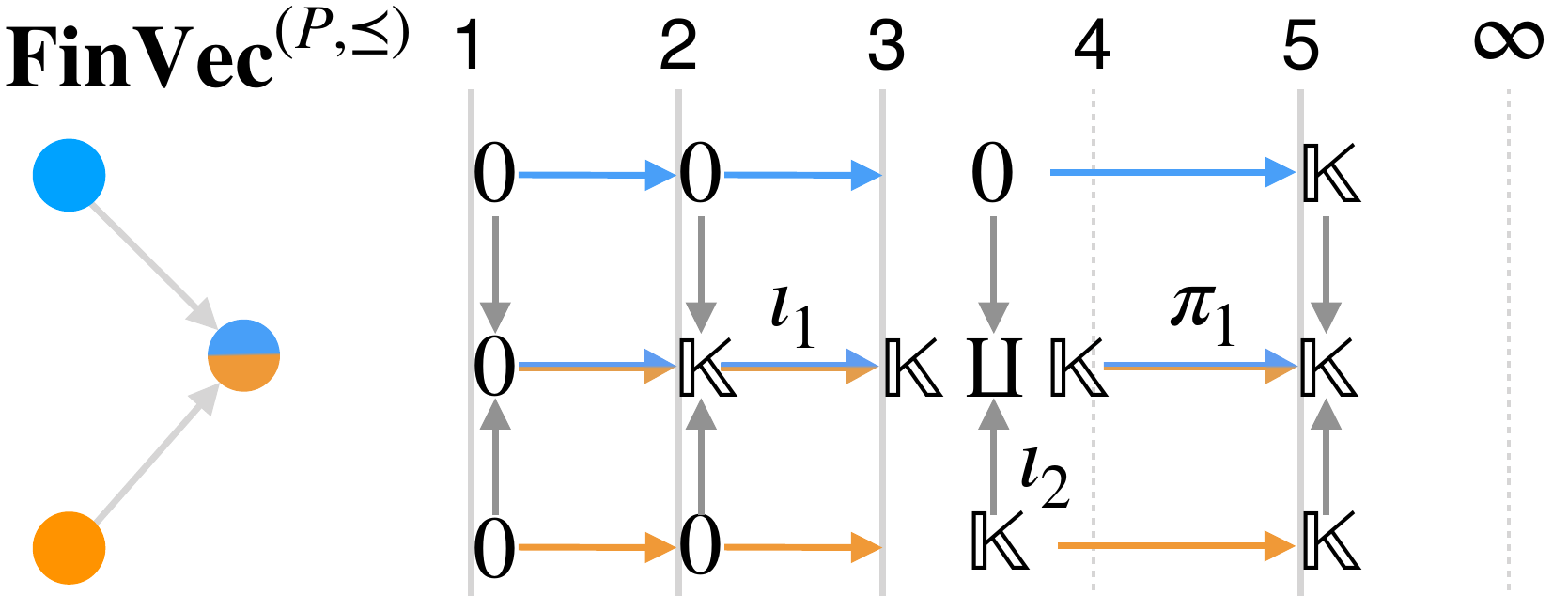}\label{fig:colored_poset_b}}
\qquad
\subfigure[][]{\includegraphics[width=.4\textwidth]{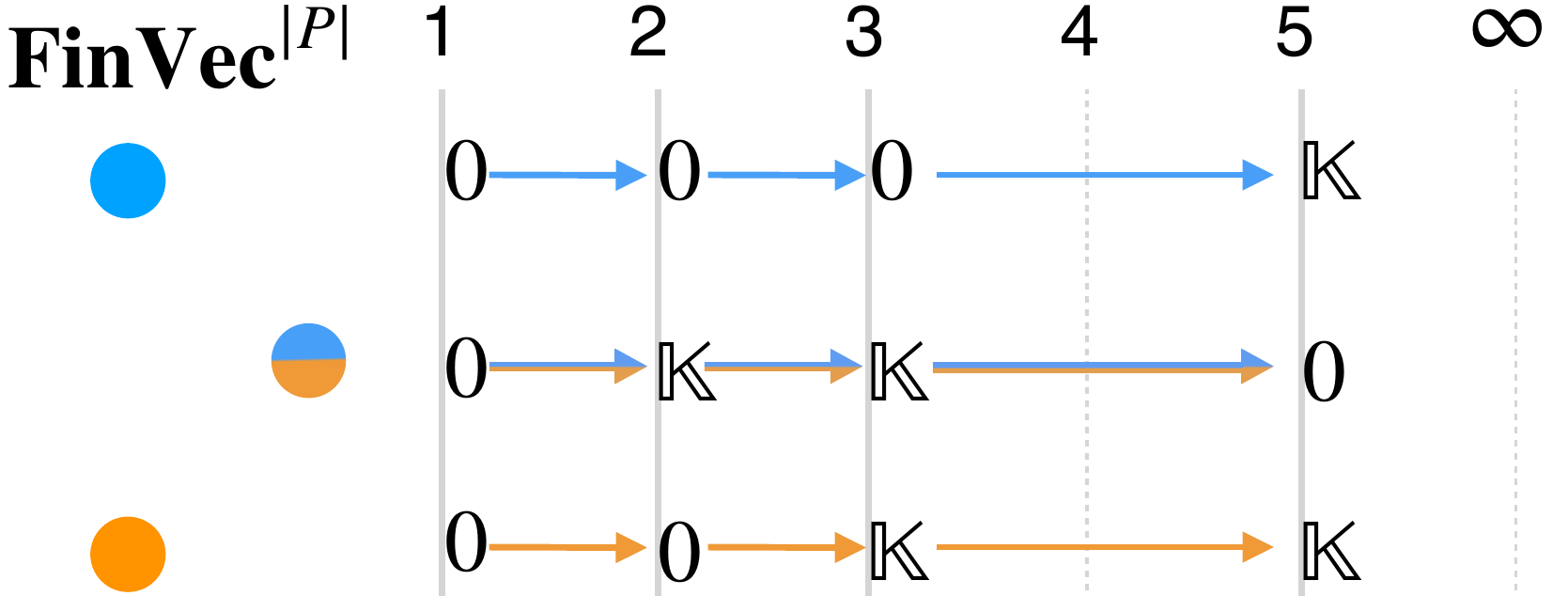}\label{fig:colored_poset_c}}

\caption{\textbf{Multicolored persistence.} We consider the colored weighted graph depicted in Panel (a), obtained by considering pairwise distances of points in a finite metric space, and the Vietoris-Rips filtration induced by the weight function defined on the edges. The multicolored persistence diagram in Panel (b) is obtained by considering the persistence of the cycles in the filtration along with their color, as depicted in Panel (c). Panels (d) and (e) are $\Rl$-indexed diagrams in $\mathbf{FinVec}^{(P, \preceq)}$ and $\mathbf{FinVec}^{|P|}$, respectively. Note how the indecomposable components of the diagram in Panel (e) are the ones described in~\Cref{sec:indecomposable}.}
\label{fig:colored_poset}
\end{figure}

\paragraph{Application: Vietoris-Rips filtration with labeled data}
Let $(X, d, l)$ be a finite metric space with a labeling function $l: X \to \{ l_1, \dots, l_n\}$ from $X$ to a discrete set of labels. Let $X_1, \dots, X_n$ be the subdatasets corresponding to the various labels, i.e. $X_i = l^{-1}(l_i)$. We wish to answer the following question: how do the homologies of the various $X_i$ interact with one another? Let $(P_{n}, \subseteq)$ be the poset of non-empty subsets of $\{1,\dots, n\}$ ordered by inclusion. We have a functor $(P_{n}, \subseteq) \to \mathbf{Met}$ sending $r \subseteq \{1, \dots, n \}$ to $\cup_{i \in r} X_i$. By applying the Vietoris Rips construction we obtain a $(P_{n}, \subseteq)$-indexed diagram of finite simplicial complexes. This allows us to build a multicolored persistence diagram from a labeled dataset keeping into account whether a persistent cycle originates from a single subdataset or a union. See~\cref{fig:colored_poset}.

\section{Conclusion}

Topological persistence and persistent homology allow for a deeper understanding of the high-dimensional organization of data~\cite{carlsson2008local,qaiser2016persistent,bae2017beyond}. Notably, persistence diagrams provide an encompassing view on the topological and geometrical properties of a given dataset, both as a whole and at sample level. The main limitation of these methods is their innate confinement to the category of topological spaces. In~\cite{bergomi_beyond_2019}, we described a first generalization of persistence to concrete categories, extending the persistence paradigm to the analysis of objects such as weighted graphs and quivers, without need of auxiliary topological constructions. However, the classical persistence homology can not be deduced naturally from this generalization. Specifically, whereas the coherent sampling technique defines a persistence function from a set-valued functor (e.g. the connected components), higher homology functors are naturally vector space-valued.

The proposed framework further generalizes both the classical and the concrete category-based persistence. We captured the essential properties of the cardinality function in $\mathbf{FinSet}$ and the dimension function in $\mathbf{FinVec}$ upon which the theory of \textit{size} and \textit{persistent Betti numbers} is built. This led us to the definition of \textit{ranked category}: a regular category equipped with an integer-valued \textit{rank function} defined on its objects. We provide strategies to build such functions as \textit{fiber-wise rank functions}. As special cases of fiber-wise ranks we recover both the cardinality of sets and dimension of vector spaces, as well as the \textit{length} function in the general case of Abelian categories of finite length.
Finally, we show how categorical persistence functions can be built from rank functions, generalizing the construction of coherent sampling introduced in~\cite{bergomi_beyond_2019}.

We provide definitions and more general proofs of the main results in classical and concrete category-based persistence. We define cornerpoints, their multiplicity and thus introduce a general paradigm to build persistence diagrams. We describe the structure of persistence modules and characterize their irreducible components in the semisimple case. These results allow us to define and discuss the interleaving and bottleneck distances, proving the stability of persistence diagrams in our framework. As finite dimensional vector spaces are an Abelian, semisimple category with essentially one simple object, we determine which of these hypotheses are needed for the classical results to hold in the generalized framework.

Our definitions are, to a large extent, preserved by functors. In particular, given two regular categories and a regular functor between them, a rank function on the target category induces a rank function on the source category. The same, without the regularity assumption, holds for any categorical persistence function.
$\Rl$-indexed diagrams, as well as $\epsilon$-interleavings between them, are preserved by arbitrary functors. As a general strategy, we apply functors to move from $\Rl$-indexed diagrams in arbitrary categories to $\Rl$-indexed diagrams in categories where the interleaving distance and the bottleneck distance are equal. In particular, this allows one to define chains of inequalities of interleaving distances in coarser and coarser categories. We observe how, in our setting, the generalized definition of filtration gets freer than the classical one, by not requiring the functions between sublevels to be monomorphisms.

The target categories of choice in the classical approach to persistence ($\mathbf{FinSet}$ or $\mathbf{FinVec}$) offer a clear correspondence between classes of isomorphism of objects and natural numbers, namely cardinality and dimension. To be able to deal with \textit{richer} categories, we develop the concept of coloring, which allows us to recover the equality between interleaving and multicolored bottleneck distance.

Finally, we discuss and exemplify \textit{via} toy examples several applications. In the Abelian semisimple case, we explicitly study the multicolored persistence and build the associated persistence diagram in the case of filtered simplicial complexes or point clouds with a group action. As much of the interest in persistent homology comes from its applications on real world data, we explore applications to point cloud data, where the extra structure is given by labels. Such datasets are routinely used in the training and testing of machine learning models on classification problems. Multicolored persistence naturally defines a topological notion of similarity of two datasets that keeps into account the labeling information. We speculate that such measure of similarity may be used to qualitatively assess the performance of machine learning models on classification problems.

\bibliographystyle{plain}
\bibliography{GroupPersistence}
\newpage
\begin{appendix}
  \numberwithin{definition}{subsection}
  \numberwithin{example}{subsection}
  \numberwithin{lemma}{subsection}
  \numberwithin{remark}{subsection}
  \input{supplement}
\end{appendix}

\end{document}

%% file: supplement.tex
\section{Basic definitions and results}
\label{sec:appendix}

The aim of this section is to provide basic definitions of category theory to the unfamiliar reader.

\subsection{Properties of morphisms}

\begin{definition}[Epimorphism]\label{def:epimorphism}
  Consider $X,Y\in \Obj(\Cat)$. $f:X \to Y$ is an epimorphism if given the following diagram
\[
\begin{tikzcd}
X \ar[r,"f"]
&
Y \ar[r,shift left=.75ex,"g"]
  \ar[r,shift right=.75ex,swap,"h"]
&
Z
\end{tikzcd}
\]
if $g\circ f = h \circ f$ then $g = h$.
\end{definition}

\begin{example}[Epimorphism]\label{ex:epimorphism}
  Let  $X,Y\in\Obj(\mathbf{Set})$ and $f:X\rightarrow Y$, then $f$ is an epimorphism if and only if it is surjective.
\end{example}

\begin{definition}[Monomorphism]\label{def:monomorphism}
  Consider $X,Y\in \Obj(\Cat)$. $f:X \to Y$ is a monomorphism if given the following diagram
\[
\begin{tikzcd}
Z \ar[r,shift left=.75ex,"g"]
  \ar[r,shift right=.75ex,swap,"h"]
&
X \ar[r,"f"]
&
Y
\end{tikzcd}
\]
if $f\circ g = f \circ h$ then $g = h$.
\end{definition}

\begin{example}[Monomorphism]\label{ex:monomorphism}
  Let  $X,Y\in\Obj(\mathbf{Set})$ and $f:X\rightarrow Y$, then $f$ is a monomorphism if and only if it is injective.
\end{example}

\subsection{Universal objects: limits and colimits}

Many interesting objects can be defined in terms of universal properties. We will describe examples from two groups: limits (there are universal arrows to them) and colimits of diagrams (there are universal arrows from them).

\begin{definition}[Terminal object]\label{def:terminal}
  An object $\pt$ in a category $\Cat$ is called \textit{terminal} if there exists a unique morphism $x\xrightarrow{!} \pt$ for any object $x\in\Cat$. If it exists, the terminal object is unique, up to unique isomorphism. For instance, the point space $\pt$ is terminal in \textbf{Top}.
\end{definition}

\begin{definition}[Initial object]\label{def:initial}
  An object $\emptyset$ in a category $\Cat$ is \textit{initial} if for any object $x\in\Cat$ there exists a unique morphism $\emptyset\xrightarrow{!}x$.
\end{definition}

\begin{definition}[Zero object and pointed category]\label{def:zero_object}
  An object which is both initial and terminal is said zero object. A category $\Cat$ equipped with a zero object is said \textit{pointed}.
\end{definition}

\begin{example}[Zero object]\label{ex:zero_object}
  The trivial group $1$ is the zero object in $\mathbf{Grp}$. Indeed $1\hookrightarrow G \twoheadrightarrow G/G=1$, for every $G\in\mathbf{Grp}$.
  In the category of $\mathbf{Vec}_{\mathbb{K}}$ of vector spaces on the field $\mathbb{K}$, the zero object is the $0$-dimensional vector space.
\end{example}

\begin{definition}[Product]
    Let $X, Y$ be objects of a category $\Cat$. The product of $X$ and $Y$ is the object $P$, along with morphisms $\pi_X:Z\to X$ and  $\pi_Y:Z\to Y$, such that given another such $P^\prime, \pi^\prime_X, \pi^\prime_y$, we have a unique morphism $P^\prime \to P$ that makes the following diagram commute:
\[
  \begin{tikzcd}
  &
  P^\prime \ar[dr, "\pi^\prime_X"]
    \ar[d, densely dotted, "u" description]
    \ar[dl, swap,"\pi^\prime_Y"]
  \\
  X
  &
  P \ar[r, swap, "\pi_Y"]
            \ar[l, "\pi_X"]
  &
  Y
  \end{tikzcd}
\]

We denote the product $X \times Y$.
\end{definition}

\begin{definition}[Coproduct]
    Let $X, Y$ be objects of a category $\Cat$. The coproduct of $X$ and $Y$ is the object $C$, along with morphisms $\iota_X:X\to C$ and  $\iota_Y:Y\to C$, such that given another such $C^\prime, \iota^\prime_X, \iota^\prime_Y$, we have a unique morphism $C \to C^\prime$ that makes the following diagram commute:

\[
\begin{tikzcd}
&
C \ar[dr, leftarrow, "\iota^\prime_X"]
  \ar[d, leftarrow, densely dotted, "u" description]
  \ar[dl, leftarrow, swap,"\iota^\prime_Y"]
\\
X
&
X\amalg Y \ar[r, leftarrow, swap, "\iota_Y"]
          \ar[l, leftarrow, "\iota_X"]
&
Y
\end{tikzcd}
\]

We denote the coproduct $X \amalg Y$.
\end{definition}

\begin{example}[Product and coproduct]
  Let  $X,Y\in\Obj(\mathbf{Set})$. The product $X \times Y$ is simply the cartesian product. The coproduct $X \amalg Y$ is the disjoint union of $X$ and $Y$.
\end{example}

\begin{definition}[Equalizer] Let $X,Y$ be objects of $\Cat$ and consider two morphisms $X\xrightarrow{f} Y$, $X\xrightarrow{g} Y$. An object $Q$, together with a morphism $Q\xrightarrow{q} Y$ is an \textit{equalizer} if $f\circ q = g\circ q$. Moreover, the pair $(Q,q)$ must be universal,~i.e. given another coequalizer $(Q^\prime, q^\prime)$, there exists a unique morphism $Q^\prime\xrightarrow{u} Q$ such that the following diagrams commutes.
\[
\begin{tikzcd}
Q \ar[r,"q"]
&
X \ar[r,shift left=.75ex,"f"]
  \ar[r,shift right=.75ex,swap,"g"]
&
Y\\
Q^\prime \ar[ru,swap,"q^\prime"] \ar[u, densely dotted, "u" description]
\end{tikzcd}
\]
Thus, equalizers are unique up to isomorphism. Moreover, every equalizer is a monomorphism.
\end{definition}

\begin{example}[Equalizer]\label{ex:equalizer}
  Let  $A,B\in\Obj(\mathbf{Set})$ and $f,g:A\rightarrow B$, then the equalizer is
  \begin{equation*}
      \{a \in A \; | \; f(a) = g(a) \}
  \end{equation*}
\end{example}

\begin{definition}[Coequalizer] Let $X,Y$ be objects of $\Cat$ and consider two morphisms $X\xrightarrow{f} Y$, $X\xrightarrow{g} Y$. An object $Q$, together with a morphism $Y\xrightarrow{q} Q$ is a \textit{coequalizer} if $q\circ f = q\circ g$. Moreover, the pair $(Q,q)$ must be universal,~i.e. given another coequalizer $(Q^\prime, q^\prime)$, there exists a unique morphism $Q\xrightarrow{u} Q^\prime$ such that the following diagrams commutes.
\[
\begin{tikzcd}
X \ar[r,shift left=.75ex,"f"]
  \ar[r,shift right=.75ex,swap,"g"]
&
Y \ar[r,"q"] \ar[dr,swap,"q'"]
&
Q \ar[d,densely dotted, "u" description]
\\
& & Q'
\end{tikzcd}
\]
Thus, coequalizers are unique up to isomorphism. Moreover, every coequalizer is an epimorphism.
\end{definition}

\begin{example}[Coequalizer]\label{ex:coequalizer}
  Let  $A,B\in\Obj(\mathbf{Set})$ and $f,g:A\rightarrow B$, then the coequalizer is the
  quotient of $B$ with $\sim$, such that $f(x)\sim g(x)$ for every $x\in A$.
\end{example}

\begin{definition}[Finitely (co)complete category]
  A category $\Cat$ is finitely complete if it has equalizers, a terimal object and binary products.
  Analogously, a category $\Cat$ is finitely cocomplete if it has coequalizers, an initial object and binary coproducts.
\end{definition}

\begin{definition}[Pullback]\label{def:pullback}
  Let $X,Y$ and $Z$ be objects of a category $\Cat$, and $X\xrightarrow{f} Z$, $Y\xrightarrow{g} Z$ morphisms. An object $P$ and the morphisms $P\xrightarrow{p_1} X$, $P\xrightarrow{p_2} Y$ is a \textit{pullback} if the following diagram

\[
\begin{tikzcd}
    P \arrow[rightarrow]{r}{p_1} \arrow[swap, rightarrow]{d}{p_2} & X \arrow[rightarrow]{d}{f} \\
    Y \arrow[rightarrow]{r}{g} & Z
\end{tikzcd}
\]
commutes and, given another such $P^\prime, p_1^\prime, p_2^\prime$ we have a unique morphism $P^\prime \xrightarrow{u} P$ that makes the following diagram commute:

\[
\begin{tikzcd}
    P^\prime
    \arrow[drr, bend left, "p_1^\prime"]
    \arrow[swap, ddr, bend right, "p_2^\prime"]
    \arrow[dr, densely dotted, "u" description] & & \\
    & P \arrow[r, "p_1"] \arrow[swap, d, "p_2"]
    & X \arrow[d, "f"] \\
    & Y \arrow[r, "g"]
    & Z
\end{tikzcd}
\]

That is to say, the pullback is universal with respect to the diagram, and thus unique up to isomorphism. We denote it $X\times_Z Y$.
\end{definition}

\begin{remark}
    The pullback $X \times_Z Y$ is the equalizer of the natural maps $Z \rightrightarrows X \times Y$.
\end{remark}

\begin{example}[Pullback]
    Given three sets $X,Y$ and $Z$ and functions $X\xrightarrow{f} Z$, $Y\xrightarrow{g} Z$, the coproduct $X\times_Z Y$ is the subset of the cartesian product:

    \begin{equation*}
        X \times_Z Y = \{ (x,y)\; | \; x \in X,\; y \in Y \text{ and } f(x) = g(y) \}
    \end{equation*}
\end{example}

\begin{example}[Fiber]\label{def:fiber}
  In the category of sets, let $\pt$ be the terminal object. Let $f:X\rightarrow Y$ be a map between sets and $y\in Y$. The fiber over $y$ is $f^{-1}(y)\subset X$ realised by the pullback
\[
\begin{tikzcd}
    f^{-1}(y) \arrow[rightarrow]{r}{} \arrow[swap, rightarrow]{d}{} & X \arrow[rightarrow]{d}{f} \\
    \pt \arrow[rightarrow]{r}{} & Y
\end{tikzcd}
\]
\end{example}

\subsection{Regular, Abelian and semisimple categories}

\begin{definition}[Regular epimorphism]\label{def:reg_epimorphism}
  An epimorphism that is the coequalizer of a parallel pair of morphism.
\end{definition}

\begin{definition}[Regular category]\label{def:regular_cat}
  A category $\Rcat$ is \textit{regular} if the following conditions hold:
  \begin{enumerate}
    \item $\Rcat$ is finitely complete.
    \item Given $X\xrightarrow{f} Y$ a morphism and its pullback $(P, p_1, p_2)$, then the coequalizer of $p_1$ and $p_2$ exists.
    \item Given the pullback
    \[
    \begin{tikzcd}
        R \arrow[rightarrow]{r}{} \arrow[swap, rightarrow]{d}{g} & X \arrow[rightarrow]{d}{f} \\
        Z \arrow[rightarrow]{r}{} & Y
    \end{tikzcd}
    \]
    if $f$ is a regular epimorphism, so is $g$.
  \end{enumerate}
\end{definition}

We will introduce some preparatory concepts to the definition of Abelian category.

\paragraph{Kernels and cokernels}
Let $\Cat$ be a category and $\xi:X\to Y$ a morphism. If for every object $Z$ and morphisms $g,h:Z\to X$, we have $\xi g=\xi h$, then $\xi$ is said a (left) zero morphism. If $\Cat$ is pointed, i.e. it has a zero object $0$, then given two objects $X, Y$ there exists a unique zero morphism $\xi: X \to Y$ given by the composition $X \to 0 \to Y$.

\begin{definition}[Kernel]
    Let $\Cat$ be a category with zero morphism $\xi$ and $f:X\to Y$ a morphism. The kernel of $f$ is defined as the equalizer
    of $\xi$ and $f$.
\end{definition}

\begin{definition}[Cokernel]
    Let $\Cat$ be a category with zero morphism $\xi$ and $f:X\to Y$ a morphism. The cokernel of $f$ is defined as the coequalizer
    of $\xi$ and $f$.
\end{definition}

\begin{definition}[Abelian category]\label{def:abelian_cat}
  A category $\Cat$ is abelian if
  \begin{enumerate}
    \item it is pointed, i.e. $\Cat$ has a zero objet;
    \item has binary products and binary coproducts;
    \item every morphism has kernel and cokernel;
    \item each monomorphism is a kernel and each epimorphism is a cokernel.
  \end{enumerate}
\end{definition}

In an Abelian category, the binary product and binary coproduct coincide and are sometimes called biproduct. We will sometimes simply call it sum, in analogy with the sum of vector spaces.

\begin{definition}[Simple object]
    Let $\Cat$ be an Abelian category. An object $X\in \Obj(\Cat)$ is simple if its only subobjects are $0$ and $X$.
\end{definition}

\begin{lemma}[Schur Lemma]
    Given $S, S^\prime$ simple objects in an Abelian category, morphisms from $S$ to $S^\prime$ are either zero or invertible.
\end{lemma}

\begin{definition}[Semisimple category]\label{def:semisimple_cat}
An Abelian category is semisimple if all its objects are semisimple, i.e.~each object can be written as a finite sum of simple objects.
\end{definition}